\documentclass[11pt, reqno]{amsart}   

\usepackage[hmargin=4cm,vmargin=4cm]{geometry}

\usepackage{amsthm,amssymb}

\usepackage{amscd,amssymb,palatino}
\usepackage[utf8]{inputenc}
\usepackage[english]{babel}
\usepackage{etex}
\usepackage[leqno]{amsmath}
\usepackage{amssymb}
\usepackage{color}
\usepackage{shadow}
\usepackage{epsfig}
\usepackage{epic}
\usepackage{graphics}
\usepackage{graphicx}
\usepackage{psfrag}
\usepackage{calc}
\usepackage{tikz}
\usepackage{tikz-3dplot}
\usepackage{tikz-cd}
\usepackage[dvipsnames,prologue,table]{pstricks}
\usepackage{pstricks}
\usetikzlibrary{arrows,decorations,patterns,positioning,automata,shadows,fit,shapes,calc,decorations.markings,backgrounds,scopes,decorations.text}

\usepackage{tipa}

 \DeclareMathOperator{\rank}{rank}
\numberwithin{equation}{section}

\newtheorem {theorem} {Theorem}[section]

\newtheorem {lemma}[theorem]{Lemma}
\newtheorem {prop}[theorem]     {Proposition}

\newtheorem {Corollary}[theorem]  {Corollary}
\newtheorem {corollary}[theorem]       {Corollary}

\theoremstyle{definition}
\newtheorem {defi}[theorem]{Definition}
\newtheorem {Remark}[theorem]          {Remark}

\newcommand{\pr} {\smallskip\noindent{\bf Proof\,\,}}

\def\R{\mathbb{R}}

\title[An $h$-principle for embeddings transverse to a contact structure]{An $h$-principle for embeddings transverse\\ to a contact structure}

\author{Robert Cardona}
\address{Robert Cardona,
Departament de Matemàtiques i Informàtica, Universitat de Barcelona, Gran Via de les Corts Catalanes 585, 08007 Barcelona, Spain \newline \it{e-mail: robert.cardona@ub.edu }
 }

 \thanks{Robert Cardona acknowledges financial support from the Margarita Salas postdoctoral contract financed by the European Union-NextGenerationEU, as well as from the LabEx IRMIA and the Universit\'e de Strasbourg. This work was partially supported by the AEI grant PID2019-103849GB-I00 / AEI / 10.13039/501100011033 and the AGAUR grant 2021SGR00603. }

\author{Francisco Presas} \address{Francisco Presas, Instituto de Ciencias Matem\'aticas-ICMAT, C/ Nicol\'{a}s Cabrera, nº 13-15 Campus de Cantoblanco, Universidad Aut\'{o}noma de Madrid,
28049 Madrid, Spain \it{e-mail: fpresas@icmat.es} }
\thanks{Francisco Presas is supported by the grant reference number PID2019-108936GB-C21 (MINECO/FEDER). This work was partially supported by the ICMAT--Severo Ochoa grant CEX2019-000904-S.}

\begin{document}

\begin{abstract}
Given a class of embeddings into a contact or a symplectic manifold, we give a sufficient condition, that we call isocontact or isosymplectic realization, for this class to satisfy a general $h$-principle. The flexibility follows from the $h$-principles for isocontact and isosymplectic embeddings, it provides a framework for classical results, and we give two new applications. Our main result is that embeddings transverse to a contact structure satisfy a full $h$-principle in two cases: if the complement of the embedding is overtwisted, or when the intersection of the image of the formal derivative with the contact structure is strictly contained in a proper symplectic subbundle. We illustrate the general framework on symplectic manifolds by studying the universality of Hamiltonian dynamics on regular level sets via a class of embeddings.
\end{abstract}

\maketitle

\section{Introduction}

The study of geometric embeddings in contact and symplectic topology has a long history. The flexibility of a given class of embeddings is formalized by proving that they satisfy an $h$-principle \cite{hprinc}: their existence is reduced to algebraic topology. Gromov's foundational work \cite{G} provided several examples of classes of embeddings (and immersions) whose flexibility is covered by the classical instances of the $h$-principle: convex integration, the holonomic approximation lemma and the micro-flexibility lemma. We refer to Eliashberg-Mishachev's book \cite{hprinc} for a modern account. These examples include contact and symplectic embeddings, isocontact and isosymplectic embeddings, or subcritical isotropic embeddings.  Since then, the study of embedded submanifolds in contact and symplectic topology has been an important feature of the field. Let us particularly mention the existence of codimension two symplectic submanifolds via Donaldson's theory \cite{Don96} or Seiberg-Witten theory \cite{Tau}, and its contact analog of the existence of codimension two contact submanifolds \cite{IMP,Gir}.

In this work, we are interested in setting up a common framework that allows us to understand the flexibility of some of these classes of embeddings as particular instances of a more general construction. This new language allows us to prove some new results about embeddings in contact and symplectic geometry. Throughout the whole article, we will restrict for simplicity to smooth manifolds (and submanifolds) that are connected and orientable. Unless otherwise specified in a statement, our results hold for embeddings of closed manifolds or embeddings of open manifolds onto locally closed submanifolds of the target manifold. The target contact or symplectic manifolds that we consider are not necessarily closed.

The inspiration comes from iso-Reeb embeddings, introduced in \cite{CMPP1}.  We give a sufficient condition for a partial differential relation $\mathcal{P}_{\xi}$ that is defined over embeddings into a given (but arbitrary) contact manifold $(M,\xi)$ to satisfy a full $h$--principle. We denote by $P_{\xi}$ the relation because it will be a condition imposed to the differential of the embedding that is related to the ambient contact structure $\xi$. The conditions that we impose can be roughly summarized in the following two requirements: (for more precise definitions see Section \ref{s:isoc})
\begin{itemize}
\item We impose that given a genuine isocontact embedding (respectively formal) $\Psi\colon (V, \tilde \xi) \to (M, \xi)$, if we have a genuine solution (respectively formal) $\phi\colon N \to V$ for the relation $\mathcal{P_ {\tilde \xi}}$, then the embedding  $\Psi \circ \phi$ is a genuine solution (respectively formal) for the relation $\mathcal P_{\xi}$. This is a natural functorial property for the relation. Such a relation is thus called {\em well behaved under inclusions relation}.
\item Given a formal embedding, there is a germ of contact submanifold $(E,\xi')$ along $N$ in such a way that we can formally deform the embedding into a new one, which is a composition of an embedding satisfying the relation for $(E,\xi')$ and the inclusion of $E$ in $M$ endowed with a formal isocontact embedding structure of $(E,\xi')$ into the ambient contact manifold $(M,\xi)$. We say that a relation satisfying this condition admits an \emph{isocontact realization}.
\end{itemize}
In a nutshell, instead of looking for an embedding compatible with the relation, we look for an intermediate submanifold with a concrete contact structure $\xi'$ for which the initial embedding satisfies $P_{\xi'}$ and then reduce the problem to finding an isocontact embedding of $(E,\xi')$ into the ambient contact manifold.

In the symplectic case, the careful reader may fill the equivalent approach just by guessing the definitions of well-behaved under inclusions relation and the natural condition for a property on a symplectic manifold having an isosympletic realization. 

The advantage of this strategy is that the first step (finding an isocontact or isosymplectic realization) might be easier than the initial problem, and the second step (flexibility of isocontact or isosymplectic embeddings) can be dealt with using known $h$-principles. There are two available statements for isocontact embeddings: one for embeddings with an overtwisted contact target, and one for embeddings of codimension greater or equal than two  (open case), or four (closed case). Respectively, we have a statement available in the symplectic case for embeddings of codimension two (open case) or four (closed case).

Once the definitions are set up in Section \ref{s:isoc}, we will outline how it covers some examples in the literature, like subcritical isotropic embeddings into contact or symplectic manifolds, Legendrian embeddings whose complement is overtwisted \cite{MNPS}, more recently iso-Reeb embeddings \cite{CMPP1}. Another simple consequence is an $h$-principle for codimension two isocontact embeddings with an overtwisted complement (Corollary \ref{cor:is2}). The techniques to show that a class of embeddings falls into the general framework can be very different, and it is the main difficulty when looking for new applications.\\

 We use our general approach to study two new classes of embeddings, one class in contact geometry and one class in symplectic geometry. The first class is the class of embeddings into a contact manifold $(M,\xi)$ that are transverse to the contact structure. Immersions of positive codimension transverse to a contact structure are known to satisfy a full $C^0$-dense $h$-principle \cite{G}, see also \cite[Theorem 14.2.2]{hprinc}. We show that there are subclasses of transverse embeddings that satisfy the $h$-principle. Recall that a contact manifold is overtwisted if there is a piecewise linear embedding of a codimension one disk (the overtwisted disk) with a specific germ of a contact structure on each connected component (see \cite{BEM}). We say that an embedding has an overtwisted complement if the complement of the embedding is an overtwisted contact manifold. In parametric families, it means that the whole family is disjoint with a fixed or a parametric family of overtwisted disks (see Remark~\ref{rm:OTcomp}).  New ideas are required to show that transverse embeddings adhere to our general setting, thus obtaining the main result of this work.

\begin{theorem}\label{thm:maintransv}
Let $(M,\xi)$ be a contact manifold. Embeddings transverse to the contact structure satisfy an $h$-principle in the following cases:
\begin{itemize}
\item the embedding has an overtwisted complement,
\item the formal embedding is \it{small}. 
\end{itemize}
The $h$-principle holds in the parametric, relative to the domain and relative to the parameter versions. In the small case, it is also $C^0$-dense.
\end{theorem}
 See Definition \ref{def:smalltransv} for the condition of being small, which requires that the intersection of the image of the formal derivative with the contact structure is strictly contained in a proper symplectic subbundle. Informally, we are able to place a new contact submanifold of positive codimension as an intermediate one, this immediately implies that we have a codimension bigger or equal than two. A sample application of the flexibility of small transverse embeddings is the following:
\begin{corollary}\label{cor:corsmall}
Let $(M,\xi)$ be a contact manifold, and $e\colon N\longrightarrow (M,\xi)$ a formal transverse embedding. Assume that $(M,\xi)$ is embedded as a contact submanifold on a higher dimensional contact manifold $i\colon (M,\xi) \rightarrow (W, \xi')$. The embedding $i\circ e\colon N\longrightarrow (W, \xi')$, which inherits a formal transverse embedding structure from $e$, is isotopic to a $C^0$-close embedding transverse to $\xi'$.
\end{corollary}
The same holds for parametric families of formal transverse embeddings in $(M,\xi)$, which can be made transverse parametrically and relative to the parameter and domain inside $(W,\xi')$. Given an arbitrary $(M,\xi)$, the simplest example of $W$ is a product of $M$ with a Liouville domain.\\

The second application of our general framework concerns Hamiltonian dynamics in symplectic manifolds. In recent works \cite{CMPP1, T1, T2, TdL}, the universality properties of Hamiltonian dynamics and ideal hydrodynamics have been investigated. The punchline of ``dynamical universality" is understanding whether an arbitrary smooth non-vanishing vector field on a closed manifold can be embedded as an invariant subsystem of a concrete class of dynamical systems (e.g. a Hamiltonian vector field along a compact regular energy level set) on a higher dimensional manifold. For example, a vector field can be embedded in a Reeb vector field in a high enough dimensional contact sphere if and only if it is geodesible \cite{CMPP1}. Similarly, a vector field can be embedded into the dynamics of a potential well as long as it admits a strongly adapted one-form \cite{T1}. It is interesting to establish a dimensional lower bound for the manifold where the target dynamical system is defined and to try to prescribe the ambient manifold or associated geometric structure (symplectic or contact for example). In the last section, we investigate the universality properties of Hamiltonian systems along regular energy level sets. More precisely, let $X$ be an arbitrary non-vanishing vector field on a manifold $N$ and $(W,\omega)$ be a symplectic manifold. Given an embedding of $N$ into $W$, is there a function $H\in C^\infty(W)$ for which $N$ lies in a regular level set of $H$ and whose Hamiltonian vector field restricts along $N$ as $X$? We call such an embedding ``Hamiltonian", in analogy to Reeb embeddings introduced in \cite{CMPP1}. We adapt the techniques introduced in \cite{CMPP1} to the setting of Hamiltonian flows and exhibit an application of our general framework in symplectic geometry.

\begin{theorem}\label{thm:mainuniv}
Let $N$ be a closed manifold and $X$ a smooth non-vanishing vector field on $N$. Let $2m$ the smallest even integer in $\{3\dim N+3, 3\dim N+4\}$ and $S^{2m-1}$ the unit sphere in $(\mathbb{R}^{2m},\omega_{std})$. Then there exists an embedding $e\colon N\rightarrow S^{2m-1}$, an embedding $f\colon S^{2m-1} \rightarrow \mathbb{R}^{2m}$ that is a $C^0$-perturbation of the standard unit sphere in $\mathbb{R}^{2m}$, and a compactly supported function $H\in C^{\infty}(\mathbb{R}^{2m})$ such that 
\begin{itemize}
\item $f(S^{2m-1})$ is a regular energy level set of $H$,
\item the Hamiltonian vector field $X_H$ of $H$ satisfies $X_H|_{f(e(N))}=f_*e_*X$.
\end{itemize}
\end{theorem}
This theorem exhibits the dynamical universality of Hamiltonian systems and shows that even $C^0$-perturbations of the standard unit sphere in $(\R^{2m},\omega_{std})$ are flexible enough to contain, in an invariant submanifold of high enough codimension, arbitrary dynamics (e.g. hyperbolic strange attractors, minimal flows, vector fields all whose orbits are periodic but whose maximal period admits no upper bound, etc...). We consider the concrete situation of perturbations of the standard unit sphere for the sake of exposition. A completely general statement for any embedded submanifold $N$ in a hypersurface of a symplectic manifold of dimension $2m$ could be stated using our general $h$-principle framework, see Remark~\ref{rm:generalHam}. We point out that the ambient dimension that we obtain is needed for our proofs but is not, at least in general, sharp. On the other hand if one prescribes not only the dynamics but also the restriction of the ambient symplectic structure to the embedded hypersurface $N$ (which we need to do in our proofs), then examples for which the dimension is almost sharp exist. These can be constructed by arguing exactly as in the examples discussed in \cite[Proposition 5.14]{CMPP1} for the case of Reeb embeddings.  \\

The paper is organized as follows. We start by reviewing some background material on the $h$-principle in contact and symplectic geometry in Section \ref{s:prel}. We give sufficient conditions for a class of embeddings into a contact or symplectic manifold for it to satisfy some versions of the $h$-principle in Section \ref{s:isoc}. In Section \ref{s:trans} we study transverse embeddings into contact manifolds and apply our previous discussion to prove Theorem \ref{thm:maintransv} and deduce Corollary \ref{cor:corsmall}. In Section \ref{s:Hamst}, we give motivating examples for the study of Hamiltonian embeddings and give a preliminary result for Hamiltonian structures. Finally, we define and analyze Hamiltonian embeddings in Section \ref{s:Hamemb}, and prove Theorem \ref{thm:mainuniv}. The Appendix contains several technical lemmas needed throughout the paper. \\

\textbf{Acknowledgements:} The authors would like to thank Eduardo Fernández and Álvaro del Pino for their useful comments, and the referees for their valuable suggestions which improved the quality of the paper. The first author is grateful to Daniel Peralta-Salas for his invitation to the Instituto de Ciencias Matemáticas in September 2021.

\section{Preliminaries}\label{s:prel}

In this section, we briefly summarize some differential relations in contact and symplectic geometry that satisfy the $h$-principle.

Given a smooth fibration $X$ over a manifold $V$, a subset $\mathcal{R} \subset J^r(X)$ of the jet space of order $r$ is called a partial differential relation of order $r$. The set of sections of $(V,J^r(X))$ that lie in $\mathcal{R}$ is denoted by $\operatorname{Sec}_\mathcal{R}(V,J^r(X))$ and is called the space of formal solutions.  The sections of $X$ whose $r$-jet extension is a formal solution are called holonomic (or genuine) solutions. We say that $\mathcal{R}$ satisfies an existence $h$-principle if every formal solution is homotopic (through formal solutions) to a genuine solution. Other versions of the $h$-principle are:
\begin{itemize}
\item  \emph{relative parametric $h$-principle}: Fix a closed subset $C \subset K$, where $K$ is any compact parameter space. Assume we have a family of formal solutions $\sigma_k, k\in K$ such that $\sigma_k$ is a holonomic solution for $k\in C$. Then there exists a parametric family of formal solutions $\tilde \sigma_{k,t}, t\in [0,1]$ such that $\tilde \sigma_{k,0}=\sigma_k$, $\tilde \sigma_{k,1}$ are holonomic solutions and moreover $\tilde \sigma_{k,t} = \sigma_k$ for $t\in [0,1]$ and $k\in C$.

\item  \emph{$h$-principle relative to the domain}: For any closed subset $D\subset V$, assume we have a formal solution $\sigma$ that is holonomic in an open neighborhood $U$ of $D$. Then there exists a family of formal solutions $\sigma_t, t\in [0,1]$ such that $\sigma_0=\sigma$, $\sigma_1$ is holonomic and $\sigma_t|_U=\sigma|_U$ for all $t$.

\item  \emph{$C^0$-dense $h$-principle}: if any formal solution $\sigma \colon V \rightarrow J^r(X)$ is homotopic to a holonomic solution $j^r(s)$ such that $s$ can be chosen arbitrarily $C^0$-close to the projection of $\sigma$ to $J^0(X)$.

\end{itemize}

A partial differential relation that satisfies a relative to the parameter, relative to the domain $h$-principle induces a weak homotopy equivalence by the inclusion of genuine solutions into formal solutions of $\mathcal{R}$. We say in this case that it satisfies a full $h$-principle.

\subsection{Flexibility in contact and symplectic geometry}

 The first flexible object that will be of interest in this work is that of overtwisted contact structures. Any contact structure is either overtwisted or non-overtwisted (usually called a tight contact structure). For this relation, a genuine solution is an overtwisted contact structure in the ambient manifold, and a formal solution is defined as an almost contact structure, also called a formal contact structure.
\begin{defi}
A formal contact structure (or almost contact structure) is a pair $(\alpha,\omega)\in \Omega^1(M)\times \Omega^2(M)$ such that $\alpha\wedge \omega^n$ is a volume form.
\end{defi}
 Note that in particular the hyperplane distribution $\xi=\ker \alpha$ equipped with $\omega|_{\xi}$ defines a symplectic vector bundle. We will say that two formal contact structures $(\alpha_0,\omega_0)$ and $(\alpha_1,\omega_1)$ are homotopic if there is an homotopy of pairs $(\alpha_t,\omega_t)$ such that $\alpha_t\wedge \omega_t^n$ is a volume form for each $t\in [0,1]$. Equivalently, the symplectic vector bundles $(\xi_0=\ker \alpha_0,\omega_0|_{\xi_0})$ and $(\xi_1=\ker \alpha_1,\omega_1|_{\xi_1})$ are homotopic as symplectic hyperplane distributions of $TM$. The main theorem is that overtwisted contact structures satisfy a full $h$-principle.

\begin{theorem}\cite{BEM, El} \label{thm:BEM}
Overtwisted contact structures (with a fixed overtwisted disk) satisfy a full $h$-principle.
\end{theorem}

The other two relations that will be very important in our discussion are isocontact and isosymplectic embeddings. Recall that a map $f\colon (N,\xi_N=\ker \alpha_N)\longrightarrow (M,\xi_M=\ker \alpha_M)$ between contact manifolds is called isocontact if $f^*\alpha_M=g.\alpha_N$ where $g \in C^\infty(N)$ is an everywhere positive function. A bundle monomorphism $F\colon TN\rightarrow TM$ is called isocontact if $\xi_N= F^{-1}(\xi_M)$ and $F$ induces a conformally symplectic map with respect to the conformal symplectic structures $CS(\xi_N)$ and $CS(\xi_M)$.  A formal isocontact embedding is an embedding $f\colon (N,\xi_N) \longrightarrow (M,\xi_M)$ and a family of bundle monomorphisms $F_t\colon TN \longrightarrow TM$ covering $f$ such that $F_0=df$ ad $F_1$ is an isocontact bundle monomorphism. The classical $h$-principle for isocontact embeddings, proved by Gromov \cite{G}, works in codimension four (or two in the open case).

\begin{theorem}\label{thm:isoc}
Isocontact embeddings of $(N,\xi_N)$ into $(M,\xi_M)$ satisfy a full $h$-principle in the following cases: 
\begin{itemize}
\item $N$ is closed and $\operatorname{cod}N\geq 4$. In this case, the $h$-principle is also $C^0$-dense.
\item $N$ is open and $\operatorname{cod} N \geq 2$. Given a core $N_0\subset $, we can choose the genuine solution to be $C^0$-close to the formal solution near $N_0$. Furthermore, the $h$-principle is relative to a domain $A$ as long as $N\setminus A$ has only open components.
\end{itemize}
\end{theorem}
A polyhedron $N_0 \subset N$ is called a core of an open manifold $N$ if, for any arbitrarily small neighborhood $U$ of $N_0$, there is an isotopy fixed at $N_0$ which brings $U$ to $N$.

We say that an embedding in a closed contact manifold has an overtwisted complement if each connected component of its complement is an overtwisted contact manifold. In this case, the $h$-principle for overtwisted contact manifolds yields an improved isocontact embedding theorem, which is just the parametric version of \cite[Corollary 1.4]{BEM}.

\begin{Corollary}\label{cor:BEM}
Codimension zero isocontact embeddings of open manifolds with an overtwisted complement satisfy an $h$-principle that is parametric, relative to the parameter, and relative to any domain $A$ such that $M\setminus A$ is overtwisted.
\end{Corollary}

More recently, an existence $h$-principle was established for codimension two isocontact embeddings of closed manifolds into any contact manifold \cite{CPP}, thus improving Theorem \ref{thm:isoc}.\\

On the symplectic side, a map $f\colon (V,\omega')\rightarrow (W,\omega)$ between symplectic manifolds is called isosymplectic if $f^*\omega=\omega'$. A bundle monomorphism $F\colon TV\rightarrow TW$ covering an embedding $f$ is called isosymplectic if $\omega'=F^*\omega$ \footnote{Abusing notation, we denote by $F^*\omega$ the form acting as $F^*\omega(u,v)=\omega(F(u),F(v))$ as in \cite{hprinc}.} and $f^*[\omega]=[\omega']$.  A formal isosymplectic embedding is an embedding $f\colon (V,\omega') \longrightarrow (W,\omega)$ and a family of bundle monomorphisms $F_t\colon TV \longrightarrow TW$ covering $f$ such that $F_0=df$ ad $F_1$ is an isosymplectic bundle monomorphism. The $h$-principle for isosymplectic embeddings was proved by Gromov as well \cite{G}.

\begin{theorem}\label{thm:isosymp}
Isosymplectic embeddings of $(V,\omega')$ into $(W,\omega)$ satisfy a full $h$-principle in the following cases:
\begin{itemize}
\item if $V$ is open and $\operatorname{cod}V\geq 2$. Given a core $V_0\subset V$, the genuine solution can be chosen to be $C^0$-close to the formal solution near $V_0$.
\item if $V$ is open or closed and such that $\operatorname{cod}V\geq 4$. Then the $h$-principle is also $C^0$-dense. 
\end{itemize}
\end{theorem}
\subsection{Flexibility of even contact and Hamiltonian structures}

We recall here two other geometric structures that satisfy the $h$-principle, as proved by McDuff \cite{McD} using the technique of convex integration.

\begin{defi}
An even contact structure on a manifold $M$ of dimension $2n$ is a hyperplane field defined by a one-form $\alpha$ such that $\alpha\wedge d\alpha^{n-1}\neq 0$. The one-form $\alpha$ is called an even contact form.
\end{defi}
Their formal counterpart is defined as follows.
\begin{defi}
A formal even contact structure on a manifold $M$ of dimension $2n$ is a pair $(\alpha, \beta)\in \Omega^1(M)\times \Omega^2(M)$ such that $\alpha \wedge \beta^{n-1}\neq 0$.
\end{defi}
The main theorem is that these satisfy a full $h$-principle \cite[Section 20.6]{hprinc}.
\begin{theorem}[$h$-principle for even contact structures]\label{thm:hprincEven}
Let $M$ be an even-dimensional manifold. Even contact structures satisfy all forms of the $h$-principle except the $C^0$-dense one. 
\end{theorem}
\begin{Remark}\label{rem:C0even}
The theorem proves a weak $C^0$-dense property as well. The formal even contact structure $(\alpha,\beta)$ is homotopic to an even contact form $(\tilde \alpha, d\tilde \alpha)$ that can be chosen to satisfy that $\tilde \alpha$ is arbitrarily $C^0$-close to $\alpha$, even if $d\tilde \alpha$ will not in general be close to $d\alpha$. This property will be important in our application of this theorem.
\end{Remark}

Similarly, in odd-dimensional manifolds, there is an odd version of symplectic structures that satisfy a full $h$-principle.

\begin{defi}
A Hamiltonian structure (or odd-symplectic structure) on an oriented $(2m+1)$-dimensional manifold $M$ is a closed two-form $\omega$ of maximal rank.
\end{defi}
The formal counterpart does not require that the form is closed anymore. 
\begin{defi}
A formal Hamiltonian structure on an oriented $2m+1$ dimensional manifold is a two-form of maximal rank.
\end{defi}
 Hamiltonian structures satisfy a full $h$-principle as well \cite[Theorem 2.1]{McD}, see also \cite[Theorem 20.5.1]{hprinc}.
\begin{theorem}[$h$-principle for Hamiltonian structures]\label{thm:McD}
Hamiltonian structures satisfy all forms of the $h$-principle except the $C^0$-dense one. One can prescribe the cohomology class $a=[\lambda]\in H^2(M,\mathbb{R})$ of the genuine solutions, parametrically, and relative to a domain $A$ if the formal solution $\omega$ satisfies $[\omega]|_A=[\lambda]|_A$.
\end{theorem}

\section{Embeddings admitting isocontact/isosymplectic realizations}\label{s:isoc}

In this section, we study partial differential relations defined over embeddings of a manifold into either a contact or a symplectic manifold. If the relation satisfies two properties, it follows that those embeddings satisfy a very general $h$-principle.

\subsection{Embeddings in contact manifolds}

Let $(M,\xi)$ be a contact manifold, and $P_{\xi}$ be a differential relation depending on $\xi$, defined in $J^1(N,M)$ over embeddings of $N$ into $M$. A formal solution to such relation is a pair $(e, F_t)$ where $e$ is an embedding and $F_t\colon TN\longrightarrow TM$ is a family of monomorphisms covering $e$ (a ``tangential rotation") such that $F_1$ satisfies $P_{\xi}$.

\begin{defi}\label{def:isoext}
The property $P_\xi$ admits an \textit{isocontact realization} of codimension $2k$ if the following holds. Given a formal (respectively genuine) solution $(e,F_t)$ of $P_\xi$, there is an open submanifold $E$ of codimension $2k$ for which $N$ is a core, a contact structure $\xi'$ on $E$ such that $N\subset E$ satisfies $P_{\xi'}$, and some family of bundle monomorphisms $R_t\colon TE\rightarrow TM$ endowing $(E,\xi')$ with a formal (respectively genuine) isocontact embedding structure into $(M,\xi)$. Furthermore, we require that $R_t|_N$ is homotopic to $F_t$ through formal solutions.
\end{defi} 
Intuitively, the isocontact realization property says that near the embedding of $N$, and up to homotopy of formal contact structures starting at the ambient contact structure, there exists a germ of contact submanifold $E \supset N$ on which $N$ is a genuine solution. In our proofs $E$ is a neighborhood of the zero section of a disk subbundle of the normal bundle of $N$ in $M$. The other property that we need is the following.
\begin{defi}
A property $P_\xi$  defined on embeddings of $N$ \textit{behaves well under inclusions} if given a formal (respectively genuine) isocontact embedding of $(E,\xi')$ into $(M,\xi)$, any formal (respectively genuine) solution of $P_{\xi'}$ naturally extends to a formal (respectively genuine) solution of $P_\xi$. 
\end{defi}
In both cases, we will require that these properties are satisfied parametrically. The next general theorem is inspired by the $h$-principle for iso-Reeb embeddings \cite{CMPP1} and the proof of the $h$-principle for generalized isocontact immersions \cite[Section 16.2]{hprinc}.

\begin{theorem}\label{thm:hprincPxi}
Let $P_\xi$  be a property that admits an isocontact realization and behaves well under inclusions. Embeddings satisfying a property $P_\xi$  satisfy a parametric and relative to the parameter $h$-principle in the following two cases:
\begin{itemize}
\item We consider parametric families of embeddings whose complement is overtwisted.
\item The isocontact realization is of codimension two at least. In this case, the $h$-principle is also $C^0$-dense.
\end{itemize}
The $h$-principle is also relative to the domain if given a formal solution of $P_\xi$ that is already genuine near a closed subset $A\subset N$, there is some $R_t$ as in Definition \ref{def:isoext} that is constant in a neighborhood of $A$ inside $E$.
\end{theorem}

\begin{Remark}\label{rm:OTcomp}
We say that a parametric family of embeddings has an overtwisted complement if they are disjoint with a fixed overtwisted disk (in each connected component of the complement of the embedding). One can also require the weaker condition that a parametric family avoids a continuous parametric family of overtwisted disks as explained in \cite[Appendix A.1]{CdPP}. 
\end{Remark}

\begin{proof}
We will only prove the non-parametric case, the proof readily generalizes to the parametric case.\\

\textbf{Case 1: The complement of the embedding is overtwisted.} Let $e\colon N\longrightarrow (M,\xi)$ be a formal solution of $P_\xi$, we have a family of monomorphisms 
$$ F_t\colon TN \longrightarrow TM $$
such that $F_0=de$ and $F_1(TN)$ satisfies $P_\xi$. Let us first prove that $e$ is isotopic to another embedding $e_1$ which is a genuine solution of $P_\xi$. We will then check that this isotopy is through formal solutions of $P_\xi$.
 
By hypothesis, there is an open submanifold $E$ containing $e(N)$, equipped with a contact structure $\xi'=\ker \alpha'$ for which $e(N)\subset (E,\xi')$ is a genuine solution of $P_{\xi'}$ and the trivial embedding $\tilde e\colon E\longrightarrow M$ is a formal isocontact embedding with family of monomorphisms
$$ R_t\colon TE \longrightarrow TM, $$
i.e. we have $R_0=d\tilde e_0$, $\xi'=R_1^{-1}(\xi)$ and $R_t$ is a symplectic bundle map. We assume that $E$ is of codimension zero, since otherwise, we are in Case 2. Applying Corollary \ref{cor:BEM} we find an isotopy $\tilde e_s\colon (E,\xi') \longrightarrow (M,\xi)$ of formal isocontact embeddings, with $\tilde e_0=\tilde e$, endowed with a one-parametric family of homotopies of bundle monomorphisms
$$ G_{t,s}\colon TE \longrightarrow TM, $$
based at $\tilde e_s$ and such that $G_{t,0}=R_t$ and $G_{t,1}\equiv de_1$. Since $E$ is of codimension zero, $G_{t,s}$ defines a family of isomorphisms from $TM|_E=TE$ to $TM|_E$, which satisfies furthermore $(G_{t,1})^{-1}(\xi)= \xi'$ as symplectic bundles. Define $e_s=\tilde e_s|_N\colon N \rightarrow (M,\xi)$. Since $N$ inside $(E,\xi')$ satisfies $P_{\xi'}$ and $P_\xi$ behaves well under inclusions, we deduce that $G_{t,s}|_{TN}$ endows $e_s$ with a structure of formal solution to $P_\xi$. In addition, the pair $(e_1, G_{t,1}|_TN)$ is a genuine solution of $P_\xi$, since it corresponds to the natural extension of the genuine solution $e_1$ restricted in the target to $(\tilde e_1(E), \tilde e_1^*\xi=\xi')$. Finally, by hypothesis, the parametric families of isomorphisms $(e_0,G_{t,0})=(e_0,R_t|_N)$ and $(e_0, F_t)$ are homotopic as formal solutions of $P_\xi$, so $(e_0,F_t)$ and $(e_1,G_{t,1})$ are homotopic as formal solutions. This concludes the proof in the non-parametric case. 

The same proof holds parametrically since Corollary \ref{cor:BEM} is parametric. Under the assumption required in Definition \ref{def:isoext} that $R_t\equiv de$ near a closed subset $A\subset N$, we apply Corollary \ref{cor:BEM} relative to the domain (that holds because Theorem \ref{thm:BEM} is relative to the domain), thus proving that the $h$-principle is relative to the domain.\\

\textbf{Case 2: The isocontact realization is of codimension two at least.} Let $N$ be a formal solution of $P_{\xi}$ and $E$ embedded in $M$ by 
$$ \tilde e\colon E\longrightarrow M,$$
such that $\tilde e|_N=e$. By hypotheses, the embedding $e$ restricted in the target to $E$ is a genuine solution to $P_{\xi'}$ for some contact structure $\xi'$ in $E$, and $\tilde e$ is covered by a family of monomorphisms $ R_t\colon TE \longrightarrow TM$ as in Definition \ref{def:isoext}.  As before, we will prove that $e\colon N\longrightarrow M$ is isotopic through formal solutions to a genuine solution of $P_\xi$, and the proof generalizes to parametric families of embeddings.  By Theorem \ref{thm:isoc}, there is an isotopy $\tilde e_s\colon E \rightarrow M$ with $s\in [0,1]$ of formal isocontact embeddings of $(E,\xi')$ into $(M,\xi)$ such that $\tilde e_1$ is a genuine isocontact embedding. Furthermore, we can assume that $\tilde e_s$ is $C^0$-close to $\tilde e_0$ near $N$, which is a core of $E$. Hence, there is a family of monomorphisms (covering $\tilde e_s$)
$$ G_{t,s}\colon TE \longrightarrow TM, \enspace \text{for } t,s \in [0,1] $$
such that $\xi'={G_{t,1}}^{-1}(\xi)$ as a symplectic vector bundle, and we have $G_{t,0}=R_t$ and $G_{t,1}=d\tilde e_1$ for all $t \in [0,1]$. Since $e_0$ restricted to $E$ in the target satisfies $P_{\xi'}$, we deduce that $G_{t,s}|_N$ endows $\tilde e_t|_N$ with a formal solution of $P_{\xi}$ using that $P_\xi$ behaves well under inclusions. For $t=1$ we find a genuine solution $\tilde e_1|_N$ that is $C^0$-close to $e$. Finally, since $(\tilde e_0, G_{t,0})|_N=(\tilde e_0, R_t)|_N$ is formally homotopic to $(e_0, F_t)$ by hypothesis, the formal class is preserved and we conclude the proof of the existence $h$-principle. The relative to a domain $h$-principle also holds under the required hypotheses. Indeed, if $R_t \equiv d\tilde e$ on a neighborhood  $U$ of $A$ inside $E$, and $A$ lies in the core of $E$, it is clear that $E\setminus A$ has only open components. Then Theorem \ref{thm:isoc} holds relative to the domain and proves the statement.
\end{proof}

 Examples of differential relations defined over embeddings that admit an isocontact realization and behave well under inclusions include subcritical isotropic embeddings in any contact manifold \cite[Theorem 12.4.1]{hprinc}, Legendrian embeddings whose complement is overtwisted \cite{M,MNPS, BEM}, iso-Reeb embeddings in overtwisted manifolds \cite[Theorem 5.7]{CMPP1}, or small iso-Reeb embeddings into any contact manifold \cite[Theorem 5.9]{CMPP1} (see \cite[Definitions 3.6 and 5.8]{CMPP1} for the definitions of iso-Reeb and small iso-Reeb embedding). One can also define generalized isocontact embeddings (defined analogously to generalized isocontact immersions \cite[Section 16.2]{hprinc}) and their small counterpart. Those would also satisfy the hypotheses of Theorem \ref{thm:hprincPxi} by adapting the arguments in \cite{CMPP1}. An analogous of Theorem \ref{thm:hprincPxi} holds for immersions instead of embeddings too, the proof being simpler since we do not need to follow a tangential rotation but just a monomorphism of $TN$ into $TM$, see Remark \ref{rem:imm} in the next section.  The following statement is another corollary of this general framework, whose proof will set up ideas on how isocontact realizations can be constructed. It was recently applied in \cite{Av}.
\begin{corollary}\label{cor:is2}
Isocontact embeddings of codimension two closed manifolds whose complement is overtwisted into a contact manifold $(M,\xi)$ satisfy a full $h$-principle.
\end{corollary}
The proof works for isocontact embeddings of any positive codimension and of possibly open manifolds, but for codimension higher than two or open codimension two embeddings, Theorem \ref{thm:isoc} is stronger as it does not require that the complement is overtwisted.
\begin{proof}
Let us show that isocontact embeddings (of any positive codimension) define a differential relation that behaves well under inclusions and that admits an isocontact realization of codimension zero. Let $F_t\colon TN\rightarrow TM$ be a formal isocontact embedding of $(N,\xi_N)$ into $(M,\xi)$, covering an embedding $e\colon N\rightarrow M$. First, observe that if $(M,\xi)$ is a contact submanifold of a contact manifold $(W,\xi_W)$ by some inclusion map $i\colon M \rightarrow W$, then the monomorphisms $F_t$ trivially extend to monomorphisms $G_t\colon TN\rightarrow TW$ such that $F_1$ restricts to a symplectic map from $\xi_N$ to $\xi_W$. Thus $i\circ e$ inherits naturally the formal isocontact embedding structure given by $G_t$, and if the original formal embedding is genuine then $F_t\equiv de$ and $G_t$ can be chosen to be constant and equal to $di\circ e$. This shows that the differential relation behaves well under inclusions. \\

To construct the isocontact realization, extend $F_t$ to a family of isomorphisms 
$$G_t\colon TM|_N \rightarrow TM|_N,$$
and consider $(\xi_t,\omega_t)=({G_t}^{-1}(\xi),  d\alpha \circ G_t)$, where $\xi=\ker \alpha$. The formal contact structure $(\xi_1,\omega_1)$ can be assumed to satisfy $(\xi_1,\omega_1)|_{TN}=(\xi_N,d\beta)$, where $\xi_N=\ker \beta$. However, the formal contact structure $(\xi_1,\omega_1)$ fails to be genuine in a neighborhood of $N$. Take a neighborhood $E$ of $N$ diffeomorphic to a neighborhood of the zero section of the normal bundle of $N$ in $M$, choosing as fibers the symplectic orthogonal of $\xi_N$ inside $\xi_1$. This endows the bundle $E$ with the hyperplane distribution $\xi_1$ given by the direct sum of $\xi_N$ and the fibers of $\pi\colon E\rightarrow N$. An application of \cite[Lemma 16.2.2]{hprinc} shows that there exists a genuine contact structure $(\xi_2,\omega_2)$ in $E$ (up to shrinking $E$ a bit) such that $(\xi_2,\omega_2)|_{TM|_N}=(\xi_1,\omega_1)|_{TM|_N}$. This last condition trivially shows that $(\xi_2,\omega_2)$ is homotopic as a formal contact structure to $(\xi_1,\omega_1)$, and the homotopy extends to all $M$ by the homotopy extension property. Applying Lemma \ref{lem:iso}, there is a family of symplectic bundle isomorphisms 
$$\hat R_t\colon \xi_2 \longrightarrow \xi_{2-t}, \enspace t\in [0,2],$$
which extends to a family of isomorphisms
$$R_t\colon TE \longrightarrow TM. $$
By construction $R_t$ endows the inclusion of $E$ inside $M$ with a formal isocontact embedding structure of $(E,\xi_2)$ into $(M,\xi)$. Furthermore $(N,\xi_N)$ is a genuine contact submanifold of $(E,\xi_2)$, and $R_t|_{TN}$ is homotopic to $F_t$ as a formal isocontact embedding structure, thus proving that $(E,\xi_2)$ together with $R_t$ is an isocontact realization of $(e,F_t)$ of codimension zero. The statement then follows from the first item in Theorem \ref{thm:hprincPxi}.
\end{proof}

\subsection{Embeddings in symplectic manifolds}

There are analogous properties for a differential relation defined over embeddings into a symplectic manifold $(M,\omega)$. Let $P_{\omega}$ be a differential relation depending on $\omega$, defined in $J^1(N,M)$ over embeddings of $N$ into $M$. A formal solution to such relation is a pair $(e, F_t)$ where $e$ is an embedding and $F_t\colon TN\longrightarrow TM$ is a family of monomorphisms covering $e$ such that $F_1$ satisfies $P_{\omega}$.

\begin{defi}\label{def:isosympext}
The property $P_\omega$ admits an \textit{isosymplectic realization} of codimension $2k$ if the following holds. Given a formal solution (respectively genuine), there is an open submanifold $E$ of codimension $2k$ in $M$ whose core is $N$, a symplectic structure $\omega'$ on $E$ such that $N\subset E$ satisfies $P_{\omega'}$, and a family of monomorphisms $R_t\colon TE\rightarrow TM$ endowing $(E,\omega')$ with a formal (respectively genuine) isosymplectic embedding structure into $(M,\omega)$. Furthermore, the family of monomorphisms $R_t|_N$ defines a formal solution to $P_\omega$ that is homotopic to $F_t$ through formal solutions.
\end{defi} 
 Again, intuitively, the isosymplectic realization property says that up to homotopy of non-degenerate two-forms starting at the ambient symplectic structure, there exists a germ of symplectic structure on a submanifold $E \supset N$ for which the embedding is a genuine solution. As before, we will require that the relation behaves well under inclusions.
\begin{defi}
A property $P_\omega$ defined on embeddings of $N$ into $(M,\omega)$ \textit{behaves well under inclusions} if given a formal (respectively genuine) isosymplectic embedding of $(E,\omega')$ into $(M,\omega)$, any formal (respectively genuine) solution of $P_{\omega'}$ naturally extends to a formal (respectively genuine) solution of $P_\omega$.
\end{defi}
We will assume that both properties hold parametrically.

\begin{theorem}\label{thm:hprincPw}
Let $P_\omega$ be a differential relation in $J^1(N,M)$ defined over embeddings of $N$ into $(M,\omega)$ that admits an isosymplectic realization of codimension at least two and behaves well under inclusions. Then embeddings satisfying $P_\omega$ satisfy a full $C^0$-dense $h$-principle. 
Given a formal solution of $P_\omega$ that is already genuine near a closed subset $A\subset N$, if we can take $R_t$ in Definition \ref{def:isosympext} which is constant in a neighborhood of $A$ inside $E$, then the $h$-principle holds also relative to the domain.
\end{theorem}
\begin{proof}
The proof is completely analogous to the contact case. Let $e\colon N\longrightarrow M$ be a formal solution of $P_\omega$, i.e. $e$ is covered by a family of monomorphisms $F_t\colon TN \longrightarrow TM$ covering $e$ such that $F_0=de$ and $F_1$ satisfies $P_\omega$. By hypothesis, there is an open embedded submanifold $\tilde e\colon E\longrightarrow M$ endowed with a symplectic structure $\omega'$ such that $\tilde e|_N=e$ satisfies $P_{\omega'}$, and a formal isosymplectic embedding structure covering $\tilde e$
$$R_t\colon TE \longrightarrow TM, $$
such that $R_t|_N$ is homotopic to $F_t$ as a formal solution of $P_\omega$. By Theorem \ref{thm:isosymp}, there is a family of formal isosymplectic embeddings $\tilde e_s\colon (E,\omega') \longrightarrow (M,\omega)$ such that $\tilde e_1$ is a genuine isosymplectic embedding and $\tilde e_s$ is $C^0$-close to $e$ near $N$. Since $P_\omega$ behaves well under inclusions, and the inclusion of $N$ into $(E,\omega')$ satisfies $P_{\omega'}$, we deduce that $\tilde e_s|_N$ is an isotopy of formal solutions of $P_\omega$ such that $\tilde e_1|_N$ is genuine. By the same arguments that in the contact case, the formal class is preserved and the $h$-principle is relative to the domain.
\end{proof}

\begin{Remark}\label{rem:imm}
All the definitions in this section, either for contact or symplectic manifolds, can be stated for immersions instead of embeddings. Then some versions of the $h$-principle are satisfied as well, by adapting the proof of Theorems \ref{thm:hprincPxi} and \ref{thm:hprincPw}, using the $h$-principle for isocontact or isosymplectic immersions \cite[Theorems 16.1.2 and 16.4.3]{hprinc}. Examples of differential relations for immersions that fall in this framework are generalized isocontact immersions \cite[Section 16.2]{hprinc}, generalized isosymplectic immersions \cite[Section 16.5]{hprinc}, Legendrian or isotropic immersions in contact manifolds \cite[Theorem 16.1.2]{hprinc}.
\end{Remark}

\section{Transverse embeddings into contact manifolds}\label{s:trans}

Our goal is to understand when a smooth embedding into $(M,\xi)$ can be isotoped to one which is transverse to $\xi$. For this, we define first the formal counterpart of a transverse embedding.
\begin{defi}
Let $(M,\xi)$ be a contact manifold. An embedding $e\colon N \rightarrow (M,\xi)$ is a \textit{formal transverse embedding} if there is a homotopy of monomorphisms $F_t\colon TN \rightarrow TM$ covering $e$ such that $F_0=de$ and $F_1(TN) \pitchfork \xi$. 
\end{defi}

In terms of the $h$-principle, the natural existence question asks whether a given formal embedding is isotopic to a genuine transverse embedding (through formal transverse embeddings). 
\subsection{Ambient overtwisted contact manifold}

If the target manifold $(M,\xi)$ is overtwisted, we can prove a complete $h$-principle for transverse embeddings whose complement is overtwisted. 

\begin{theorem}\label{thm:OTtran}
Transverse embeddings of codimension at least one into $(M,\xi)$ whose complement is overtwisted satisfy a full $h$-principle. 
\end{theorem}

The proof of Theorem \ref{thm:OTtran} consists in showing that the differential relation $\mathcal{R}_\pitchfork(\xi)$ of embeddings transverse to $\xi$ satisfies the hypotheses of Theorem \ref{thm:hprincPxi}.

\begin{Remark}
An important step of the proof is showing that given a corank one symplectic distribution transverse along a submanifold $N$, there is a germ of contact distribution near $N$ that is homotopic to the symplectic distribution and $C^0$-close (only at the level of real distributions) to the one that we started with. The fiberwise symplectic structure of the contact germ will not, in general, be close to the original fiberwise symplectic structure. We believe that there is an alternative way of proving this fact by modifying the proof of the holonomic approximation lemma \cite{hprinc} applied to the existence of contact structures on open manifolds. We use a direct approach instead, using the $h$-principle for even contact structures.
\end{Remark} 

\begin{proof}[Proof of Theorem \ref{thm:OTtran}]
For any contact manifold $(M,\xi)$, the partial differential relation $\mathcal{R}_\pitchfork(\xi)$ clearly behaves well under inclusions, so we only need to prove that it admits an isocontact realization to apply Theorem \ref{thm:hprincPxi}. We will show first, in the non-parametric case, that $\mathcal{R}_\pitchfork(\xi)$ admits an isocontact realization of codimension zero. Let $N$ be a closed manifold of dimension $k$ and $e\colon N\rightarrow (M,\xi)$ an embedding of $N$ into an overtwisted contact manifold $(M,\xi)$ endowed with a formal transverse embedding structure $F_t\colon TN\rightarrow TM$. Extend the family of monomorphisms $F_t$ to a family of isomorphisms $G_t\colon TM \rightarrow TM$ satisfying $F_t= G_t \circ F_0$. Define $\xi_{t}=G_t^{-1}(\xi)$ and $\omega_t=d\alpha \circ G_t$. Then $(\xi_t,\omega_t)$ defines a homotopy of formal contact structures such that  $\xi_1\pitchfork TN$.\\

 To set up ideas, we will now make a simplifying assumption and later discuss the general case. This assumption is that the normal bundle of $e(N)$ admits a trivial section. This is trivially true if $N$ is of codimension one and holds as well if the codimension of the embedding is high enough. \\

\textbf{Step 1: A formal contact structure that is (genuinely) even contact along a hypersurface $\mathbf{W\supset N}$.} Let $\nu$ denote a normal bundle over $N$ lying in $\xi_1$, so that $TM|_N=TN \oplus \nu$. By assumption, there exists some line field $L$ such that $\nu=L\oplus \nu'$. Using the exponential map along the fibers of $\nu'$, we construct a small enough codimension one submanifold $W$ containing $N$ and such that $\xi_1\pitchfork W$. The symplectic distribution $\xi_1$ splits near $W$ as $\xi_1=\tilde \xi \oplus L$, where $\tilde \xi=\xi_1\cap TW$. Furthermore, the fiberwise symplectic structure $\omega_1$ has a one dimensional kernel when restricted to $\tilde \xi$, so there is another splitting $\tilde \xi= \xi'\oplus V$ where $V=(\ker \omega_1|_{\tilde \xi})$. This induces a symplectic splitting
\begin{equation}\label{eq:split}
 (\xi_1, \omega_1)|_{W}= (\tilde \xi \oplus L, \omega_1)=( \xi', \tilde \omega|_{\xi'}) \oplus (V\oplus L, \omega_1'),
\end{equation}
where $\omega_1'$ is some rank $2$ two-form that is fiberwise non-degenerate on $V\oplus L$. The pair $(\tilde \alpha, \tilde \omega)$ on $W$, where $\tilde \alpha$ is some one form such that $\ker \tilde \alpha=\tilde \xi$, defines a formal even contact structure satisfying $\tilde \xi \pitchfork N$. We will now make $\tilde \xi$ a genuine even contact structure that remains transverse to $N$.

We apply Theorem \ref{thm:hprincEven} in a neighborhood of $N$ inside $W$, and extend the resulting homotopy using the extension property, obtaining a homotopy of formal even-contact structures $(\tilde \xi_t, \tilde \omega_t)$ in $W$  with $t\in [1,2]$ such that 
\begin{itemize}
\item it is constant outside a small neighborhood of $N$ inside $W$,
\item $(\tilde \xi_1, \tilde \omega_1)=(\tilde \xi, \tilde \omega)$,
\item $(\tilde \xi_2=\ker \gamma, \tilde \omega_2=d\gamma)$ is a genuine even contact structure near $N \subset W$,
\item By Remark \ref{rem:C0even}, we can choose $\tilde \xi_2$ to be $C^0$-close to $\tilde \xi_1$ as a distribution, and thus $\tilde \xi_2\pitchfork N$.
\end{itemize}

We will now use the splitting in Equation \eqref{eq:split} to extend the family $(\xi_t,\omega_t)$ in all $TM$ for values of the parameter $t$ in $[1,2]$ . Using Lemma \ref{lem:iso} we find a family of symplectic bundle isomorphisms 
$$ \tilde R_t\colon (\tilde \xi_1, \tilde \omega_1) \longrightarrow (\tilde \xi_t, \tilde \omega_t), $$
and extend it to a family of isomorphisms
$$ R_t\colon TM|_{W} \longrightarrow TM|_{W}. $$
We might do this, for example by imposing 
$$R_t(\ker \tilde \omega_1)|_{W}=\ker \tilde \omega_t|_{W}, \enspace \text{for } t\in [1,2]$$
and keeping $L$ fixed, i.e. $R_t(L)=L$.
Decomposing $\tilde \xi_t$ as $\xi'_t \oplus (\ker \tilde \omega_t\cap \tilde \xi_t)$, we extend the family of symplectic bundles to all $TM|_{W}$ as
\begin{equation} \label{eq:extensionEC}
 (\xi_t,\omega_t)|_{W}= ( \xi_t', \tilde \omega_t|_{\xi_t'}) \oplus (R_t(\ker \omega_1)\oplus L, \omega'_t), \quad t\in [1,2],
\end{equation}
where $\omega'_t=(R_t^{-1})^*\omega_1'$. Summarizing, we have constructed a pair $(\xi_2,\omega_2)$ homotopic to $(\xi_0,\omega_0)=(\xi,d\alpha)$ throught formal contact structures such that $\xi_2= \xi_2' \oplus R_2(\ker \omega_1 \oplus L)$, $\xi_2\pitchfork N$ and $(\tilde \xi_2=\ker \gamma, \omega_2|_{\tilde \xi_2}=d\gamma)$ is a genuine even-contact structure, where by construction $\tilde \xi_2=\xi_2'\oplus \ker \omega_2|_{TW}$.\\

\textbf{Step 2: Contactization of the genuine even-contact structure by a homotopy of formal contact structures.}
 Let $U$ be the normal bundle of $W$ in $M$ with fiber $L$,
\begin{center}
\begin{tikzcd}
 L \arrow[hookrightarrow]{r}{} & U \arrow[rightarrow]{d}{\pi}\\

& W
\end{tikzcd}
\end{center}
which is equipped with the symplectic corank 1 distribution $(\xi_2,\omega_2)$. It satisfies $\xi_2|_{TW}=\tilde \xi_2=\ker \gamma$ and $\omega_2|_{TW}=d\gamma$ by construction. We will now argue as in \cite[Lemma 16.2.2]{hprinc} to find a contact structure $\xi_3=\ker \lambda$ defined near the zero section of $U$, transverse to $N$, and such that $(\xi_3,d\lambda)$ is homotopic to $(\xi_2,\omega_2)$. Choose an isomorphism of $U\cong W\times \mathbb{R}$, and denote by $z$ the coordinate in $\mathbb{R}$. We identify a neighborhood of $W$ as a neighborhood of the zero section of $U$. The symplectic bundle $\xi_2= \xi_2' \oplus \ker \tilde \omega_2 \oplus L$ is equipped with the form $\omega_2$. This form can be written as
\begin{equation} \label{eq:sympbundle}
\omega_2= dz\wedge\zeta + d\tilde \gamma,
\end{equation}
where $\tilde \gamma=\pi^*\gamma$ and $\zeta \in \Omega^1(U)$. Using this form we can construct a germ of contact structure $\xi_3=\ker \lambda$ where 
\begin{equation}\label{eq:cont}
 \lambda=\tilde \gamma + z\zeta, 
\end{equation}
which further satisfies
$$ d\lambda= d\tilde \gamma + dz\wedge \zeta + z d\zeta.$$
Up to shrinking the neighborhood $U$ of $W$, the form $\lambda$ is of contact type. The symplectic distribution $(\xi_3, d\lambda)$ coincides with $(\xi_2, \omega_2)$ along $W$, given by $z=0$, so they are clearly homotopic as symplectic bundles close to $W$. Since $\xi_3$ coincides with $\xi_2$ as a distribution along $W$, it is satisfied that $\xi_3\pitchfork N$. We end up with a family of symplectic bundles $(\xi_t,\omega_t)$ with $t\in [0,3]$ on a neighborhood of $W$. By the homotopy extension property, we have a family of symplectic vector bundles $(\xi_t,\omega_t)$ of corank $1$ defined in all $M$, with $t\in [0,3]$, such that $\xi_3 \pitchfork N$ everywhere and $\xi_3$ is of contact type on some open neighborhood $E$ of $N$. \\

\textbf{Step 3: A formal isocontact embedding.} We claim that the trivial embedding $\tilde e\colon E \longrightarrow (M,\xi)$ of the neighborhood $E$ of $N$ inside $M$ is a formal isocontact embedding (of codimension zero) of $(E,\xi_3)$ into $(M,\xi)$. We know that $\xi_3$ is formally contact homotopic to $\xi$ along $E$ by construction. Hence using Lemma \ref{lem:iso} we find a family of symplectic bundle isomorphisms 
$$ Q_t\colon \xi_3 \longrightarrow \xi_{3-t}, \quad t\in [0,3],$$
where $\xi_0=\xi$. Such a family can be extended to a family of isomorphisms 
$$ R_t\colon TE \rightarrow TM|_E$$
satisfying $R_3^{-1}(\xi_0)=\xi_3$ and $R_3$ induces a symplectic map with respect to the symplectic bundle structures in $\xi_0,\xi_3$. Hence $E$ is endowed with a formal isocontact structure, satisfying the properties that make $\mathcal{R}_\pitchfork(\xi)$ a partial differential relation that admits an isocontact realization of codimension zero. \\

\textbf{Step 4: Relative to a domain $h$-principle.} Let us check that the isocontact realization can be constructed relative to the domain as in the hypotheses of Theorem \ref{thm:hprincPxi}. Assume that there exists a closed subset $A\subset N$ on which the formal transverse embedding is already transverse. Then $F_t$ is constant and equal to $de$ close to $A$. Then following the proof, by construction $(\xi_1,\omega_1)$ is just the ambient contact structure in an open neighborhood of $A\subset N \subset W$. This implies that the formal even contact structure $(\tilde \alpha, \tilde \omega)$ is genuine in an open neighborhood of $A$ inside $W$. Since the $h$-principle for even contact structures is relative to the domain, the homotopy $(\tilde \xi_t,\tilde \omega_t), t\in [1,2]$ is constant near a possibly smaller neighborhood of $A$ inside $W$. The same holds for $(\xi_t,\omega_t), t\in [0,2]$: it is constant in a neighborhood of $A$ inside $U$. Let $B\subset N$ be a slightly smaller neighborhood containing $A$ inside the neighborhood where the homotopy is constant. Choose an open cover $B_j$ of $W$ such that each $B_j$ intersecting $A$ is contained in $B$, and each $B_j$ intersecting $B$ is included in the neighborhood where the homotopy is constant. Take a partition of the unity $\chi_j$ subordinated to $B_j$ and define
$$ \alpha'= \sum_j \chi_j \alpha_j, $$
where $\alpha_j=\alpha$ (a contact form defining $\xi$) for each $B_j$ intersecting $A$, and $\alpha_j=\lambda$ as in Equation \eqref{eq:cont} in the normal bundle of $W$ if $B_j$ does not intersect $B$. Observe that at points which lie in $B \setminus A$, the distributions defined by each one of the contact forms coincide along $TW$, so $\ker \alpha'$ will define a contact structure transverse to $N$ everywhere, even at the points where we sum combinations of $\alpha$ and $\lambda$. This shows that $\alpha'$ defines a contact structure $\xi_{3}$, which coincides with $\xi$ near $A$, homotopic to $(\xi_{2},\omega_{2})$, and transverse to $N$. It follows that the formal isocontact embedding of $(E,\xi_3)$ obtained using $R_t$ as before is genuine near $A$. We conclude by Theorem \ref{thm:hprincPxi} that the formal transverse embedding is isotopic to a genuine one relative to the domain. The same proof applies parametrically and relative to the parameter. \\

\textbf{Step 5: General case.} To cover the general case, when the normal bundle of $N$ does not necessarily admit a trivial section, we do it by constructing the germ of contact structure $\xi_3$ on all $N$ by doing it on a finite covering of $N$ by balls $B_i\subset N$ trivializing the normal bundle of $N$. The first homotopy $(\xi_t,\omega_t)$ for $t\in [0,1]$ that we constructed did not require the simplifying assumption, so we have a formal contact structure $(\xi_1,\omega_1)$ that is transverse to $N$. In a neighborhood $U_1$ of $B_1$, understood as a neighborhood of the zero section of the normal bundle of $B_1$ in $M$, we can choose a hypersurface $W_1$ containing $B_1$ and such that $\xi_1 \pitchfork W_1$. The formal contact structure induces on $W_1$ a formal even-contact structure. Arguing as in Steps 1 and 2, we make this even contact structure a genuine one in $W_1$ and construct with it a germ of contact structure homotopic to $(\xi_1,\omega_1)$, and transverse to $B_1$. By the homotopy extension property of formal solutions, this produces a formal contact structure, that we denote by $(\xi_2,\omega_2)$, homotopic to $(\xi_1,\omega_1)$, everywhere transverse to $N$, and genuinely contact near a slightly smaller ball $B_1'\subset B_1$. In $B_2$, the formal contact structure $(\xi_2,\omega_2)$ is transverse to $B_2$ and genuinely contact in $B_2\cap B_1'$. Let $W_2$ be a hypersurface in a neighborhood $U_2$ of $B_2$ containing $B_2$ and transverse to $\xi_2$. The formal contact structure induces a formal even contact structure in $W_2$. Here comes the important observation: this formal even contact structure is already genuine in a neighborhood of $B_2\cap B_1'$ inside $W_2$, because $\xi_2$ is genuinely contact in a neighborhood of $B_1'$. Using that the $h$-principle for even-contact structures is relative to the domain, we can argue exactly as before to produce a formal contact structure $(\xi_3,\omega_3)$, transverse to $N$, and genuinely contact near a slightly smaller ball $B_2'\subset B_2$ and equal to $(\xi_2,\omega_2)$ in $B_1'$. Iterating this argument, we end up with a formal contact structure $(\hat \xi, \hat \omega)$, transverse to $N$ and genuinely contact in a neighborhood $E$ of $N$. The conclusion follows arguing as in Step 3 to show that $(E, \hat \xi)$ admits a formal isocontact embedding into $(M,\xi)$ that is an isocontact realization of the formal transverse embedding.
\end{proof}

In dimension three, the previous theorem contains in particular the following folklore result, see also \cite{Et}. Let $\mathcal{K}_{\perp}^{\Delta}(M,\xi)$ denote the space of transverse knots avoiding a fixed overtwisted disk $\Delta$ and 
$\mathcal{FK}_{\perp}^{\Delta}(M,\xi)$ the corresponding formal space. Then the inclusion 
$$ i\colon \mathcal{K}_{\perp}^{\Delta}(M,\xi)\longrightarrow \mathcal{FK}_{\perp}^{\Delta}(M,\xi) $$
is a weak homotopy equivalence. This is known to be false in three dimensions if one does not fix an overtwisted disk in the complement of the parametric families of knots. In contrast with the three-dimensional case, it was recently shown that in manifolds of dimension greater than three, knots transverse to a bracket-generating distribution (in particular, to contact structures) satisfy a full $h$-principle \cite{MdP}.
\subsection{A flexible subclass in the general case}

To be able to prove a full $h$-principle without any assumption on the target contact manifold, we need to impose a condition at the formal level that is reminiscent of small formal iso-Reeb embeddings \cite{CMPP1}. This will make it possible to apply the second case of Theorem \ref{thm:hprincPxi}.

\begin{defi}\label{def:smalltransv}
An embedding $e\colon N\rightarrow (M,\xi)$ of codimension $k\geq 3$ is called a \textbf{small} formal transverse embedding if $e$ is a formal transverse embedding such that
\begin{itemize}
\item the contact structure splits as $(\xi|_N, d\alpha|_N)=(\xi' \oplus C, d\alpha|_{\xi'} + d\alpha|_C)$, where $C$ is a proper symplectic subbundle,
\item $F_1(TN) \cap \xi= F_1(TN)\cap \xi' \subsetneq \xi'$,
\end{itemize}
\end{defi}
Parametric families of small embeddings are equipped with a parametric symplectic splitting as defined above.

\begin{theorem}\label{thm:smalltrans}
Small transverse embeddings satisfy a full $h$-principle.
\end{theorem}

\begin{proof}
We have $F_1(TN)\cap \xi= F_1(TN)\cap \xi'$ and argue as in the overtwisted case to find a family of symplectic subbundles $(\xi_t, \omega_t)= (\xi'_t, \tilde \omega_t) \oplus (C, d\alpha|_C)$ such that $TN\cap \xi_t=TN \cap \xi'_t$, at time zero we have $(\xi_0,\omega_0)=(\xi,d\alpha)$ and $N \pitchfork \xi'_1$.\\

Since $TN\cap \xi'_1$ is strictly contained in $\xi_1'$, we can find an open submanifold $E$ of codimension two at least in $M$, containing $N$, using the exponential map along the fibers of a subbundle $\nu$ such that $\xi'_1|_N=(TN\cap\xi_1')\oplus \nu$. We can now apply step by step the same arguments as in the overtwisted case by considering only the symplectic subbundle $\xi'_1$ over $N$ instead of the whole $\xi_1$. We will end up with a family $(\xi'_t, \tilde \omega_t)$ with $t \in [0,2]$ of symplectic bundles defined over $E$ that satisfies:
\begin{itemize}
\item $(\xi'_0,\tilde \omega_0)= (\xi', d\alpha|_{\xi'})$,\\
\item $(\xi'_2,\tilde \omega_2)$ is a genuine contact structure on $E$ near $N$,\\
\item $N\pitchfork \xi'_2$.
\end{itemize}
From this family of symplectic bundles, we apply Lemma \ref{lem:iso} and construct a family of isomorphisms
$$ R_t\colon TE \rightarrow TE $$
such that $R_2^{-1}(\xi)=\xi_2'$ and $R_2$ induces a symplectic map with respect to the symplectic bundle structures. It extends to a family of bundle monomorphisms
$$ G_t\colon TE \rightarrow TM $$
such that $G_2^{-1}(\xi)=\xi_2'$ and $G_2$ induces a symplectic bundle map.
In particular, the trivial inclusion of $E$ into $(M,\xi)$ is endowed with a formal isocontact embedding structure from $(E,\xi_2')$ into $(M,\xi)$, which satisfies the properties making $\mathcal{R}_\pitchfork(\xi)$ a partial differential relation admitting an isocontact realization. Since $E$ is open and always of codimension $2$ at least, it follows from Theorem \ref{thm:hprincPxi} that a full $C^0$-dense $h$-principle is satisfied. Arguing as in the proof of Theorem \ref{thm:OTtran}, the $h$-principle also holds relative to the domain.
\end{proof}

Examples of formal transverse embeddings that are small are given in Corollary \ref{cor:corsmall}. Consider a formal transverse embedding $e\colon N\longrightarrow (M,\xi)$ of codimension at least one into the contact manifold $(M,\xi=\ker d\alpha)$. Let $(W,\xi')$ be a contact manifold of dimension higher than $M$ such that $i\colon (M,\xi)\rightarrow (W,\xi')$ is an isocontact embedding. The formal transverse structure of $e$, given by a family of monomorphisms $F_t\colon TN\rightarrow TM$ endows $i\circ e\colon N\longrightarrow W$ with a formal transverse structure $G_t=di\circ F_t\colon TN \rightarrow TW$. This formal transverse structure is small. Indeed, observe that there is a splitting $\xi'= \xi \oplus V$, where $V$ is the (symplectic) normal bundle of $M$ inside $W$. By construction $G_1(TN)\cap \xi' \subsetneq \xi$, and hence the formal embedding is small. Applying Theorem \ref{thm:smalltrans} we deduce Corollary \ref{cor:corsmall}.

\section{Hamiltonian energy level sets with arbitrary dynamics}\label{s:Hamst}

In this section, we present two constructions of Hamiltonian dynamics on which we can prescribe arbitrary dynamics along a low-dimensional invariant submanifold in a regular energy level set. This will motivate a general study of these invariant submanifolds in the last section via a class of embeddings. It will be an example of a differential relation to which Theorem \ref{thm:hprincPw} applies.

\subsection{The cotangent lift}\label{ss:HamEx} A classical construction, see e.g. \cite{T1}, permits realizing arbitrary dynamics as a Hamiltonian system in the cotangent bundle of the ambient manifold. To see this, let $N$ be a compact manifold and $X$ an arbitrary vector field defined on it. Equip $T^*N$ with the standard symplectic form $\omega=d\lambda$, where $\lambda$ is the Liouville one-form. Take as the Hamiltonian function $H(q,p)=p(X)$, where $q,p$ denote coordinates respectively in the base and the fibers of $T^*N$. It is easy to check that the zero section is an invariant submanifold of the Hamiltonian vector field defined by $H$ (whose flow is given by $(d\phi_t)^*$ where $\phi_t$ is the flow of $X$) and that $X_H|_{N\times \{0\}}=X$.\\

When $X$ is a non-vanishing vector field, the level set $H=0$ is a regular level set and the zero section of $T^*N$ lies on this level set. The fact that $H=0$ is regular follows easily from the fact that in coordinates one has
$$ H=\sum_{i=1}^n p_i X_i(q_1,...,q_n), $$
and since $X\neq 0$, at least one derivative with respect to $p_i$ is non-vanishing at every point. The level set $H=0$ will not be a closed submanifold in general. The following discussion will motivate the introduction of what we will call Hamiltonian embeddings. There exists an isosymplectic embedding
$$ e\colon T^*N \longrightarrow \R^{4n} $$ 
of the symplectic manifold $(T^*N,\omega)$ into $(\R^{4n},\omega_{std})$ using a classical theorem by Gromov \cite{G}.  The neighborhood of this embedded submanifold is given by its symplectic normal bundle $U$, so that there is a projection $\pi\colon U \longrightarrow e(T^*N)$. The function $\pi^*H$ defines a germ of a function $\pi^*H$ such that the Hamiltonian flow defined by such function along $\pi^*H=0$ has an invariant submanifold $e(T^*N)$ where $X_{\pi^*H}$ coincides with the previous Hamiltonian flow $X_H$ defined on $T^*N$. In particular, the flow coincides with the original arbitrary vector field $X$ along the zero section of $e(T^*N)$.  In other words, if there exists a function $f\in C^\infty(\R^{4n})$ such that $f|_U=\pi^*H$ and that admits zero as a regular value, then we have realized arbitrary dynamics on (possibly non-compact) regular level sets of Hamiltonian functions in the standard symplectic space. Of course, the existence of such a function $f$ is not clear in general.  In the next subsections, we will study this approach in much more generality from the point of view of the symplectic flexibility of what we call ``Hamiltonian embeddings". As a particular application, we will deduce that it is possible to realize arbitrary dynamics in the standard symplectic space (of high enough dimension) on a compact regular energy level set given by a $C^0$-perturbation of the standard unit sphere.

\subsection{Arbitrary dynamics on Hamiltonian structures}

As a starting point, we first show that it is possible to realize arbitrary dynamics along the kernel of a Hamiltonian structure, by taking a closed odd-dimensional manifold of high enough dimension. We will sometimes work with the line field defined by an arbitrary (non-vanishing) vector field, and so this line field is always orientable.
\begin{prop}\label{thm:HamS}
Let $N$ be any open or closed manifold and $M$ an almost contact manifold such that $\dim M\geq 3\dim N$. Fix a line field $L$ on $N$:
\begin{enumerate}
\item[(1)] Let $e\colon N \rightarrow M$ be any smooth embedding (onto a locally closed submanifold of $M$). Any exact Hamiltonian structure $\omega$  on $M$ is homotopic through exact Hamiltonian structures to an exact Hamiltonian structure $\hat \omega$ satisfying $e_*L=\ker \hat \omega|_{e(N)}$.
\end{enumerate}

\begin{enumerate}
\item[(2)] There is \emph{some} embedding $\tilde e$ of $N$ into $M$ such that for any cohomology class $a\in H^2(M)$ there is a Hamiltonian structure $\omega_a$ in $M$ such that $[\omega_a]=a$ and $\tilde e_*L=\ker \omega_a|_{\tilde e(N)}$.
\end{enumerate}
\end{prop}
The almost contact condition is required for the statement not to be empty since it is equivalent to the existence of an exact Hamiltonian structure by McDuff's theorem.
\begin{proof}
Let $\omega$ be an exact Hamiltonian structure on $M$. The idea of the proof is to homotope $\omega$ to another non-degenerate (but not necessarily closed) two-form whose kernel along $N$ is exactly $e_*L$. We produce a germ of a closed non-degenerate two-form also having $e_*L$ as kernel near $N$ and extend it by a homotopy to a global Hamiltonian structure using McDuff's $h$-principle.

 Take any codimension one distribution $\xi$ where $\omega$ is non-degenerate, thus obtaining a codimension one symplectic distribution $(\xi,\omega|_{\xi})$. On the other hand, let $\eta$ be any hyperplane distribution on $N$ transverse to $L$. Applying Proposition \ref{prop:isotr}, we find a family of symplectic codimension one distributions $(\xi_t,\hat \omega_t)$
\begin{itemize}
\item $(\xi_0,\hat \omega_0)=(\xi,\omega|_{\xi})$,
\item $(\xi_1,\hat \omega_1)$ satisfies that $\xi_1|_N\cap TN=\eta$ and $\eta$ is an isotropic subbundle.
\end{itemize}
Take a family of line fields $X_t$ such that $X_t\oplus \xi_t=TM$, $X_0=\ker \omega$ and $X_1|_N=e_*L$. This is possible because $\ker \omega$ is transverse to $\xi_0$ and $e_*L$ is transverse to $\xi_1|_N$. We define a family of maximal rank two forms $\omega_t$ such that $\omega_0=\omega$, by requiring that each $\omega_t$ coincides with $\hat \omega_t$ on $\xi_t$ and $\ker \omega_t=X_t$. \\

Let $U$ be a tubular neighborhood of $N$ in $M$, understood as a neighborhood of the zero section of the normal bundle of $N$. We will apply now a trick that we already used in the proof of Theorem \ref{thm:OTtran}. The non-degenerate two-form $\omega_1$ vanishes when restricted to the zero section of $U$, i.e. along $N$. By \ref{lem:Hamgerm} there exists a maximally non-degenerate exact two-form $\tilde \omega$ defined near the zero section and such that $\tilde \omega|_{TU|_N}=\omega_1|_{TU|_N}$. The last condition can easily be used to show that $\tilde \omega$ is homotopic to $\omega_1$ through non-degenerate two-forms. We extend the family $\omega_t$ to values $t\in [1,2]$ such that $\omega_2=\tilde \omega$. Summarizing, we have proved that there is a family of two-forms $\omega_t$ with $t \in [0,2]$ such that
\begin{enumerate}
\item $\omega_t=\omega$ for $t\in [0,\delta]$,
\item $\omega_t$ is non-degenerate for all $t$,
\item $\omega_t=\omega_2$ for $t\in [2-\delta,2]$
\item $\omega_2$ is exact near $N$,
\item $\ker \omega_2|_N=e_*L$.
\end{enumerate}

Applying Theorem \ref{thm:McD}, we find a $I^2$-parametric family of two forms $\omega_{t,s}$ such that 
\begin{enumerate}
\item $\omega_{t,s}$ is of maximal rank,
\item $\omega_{t,0}=\omega_t$,
\item $\omega_{t,1}$ is exact,
\item $\omega_{0,s}=\omega$,
\item $\omega_{1,s}=\omega_2$ near $N$.
\end{enumerate}

In particular $\omega_{t,1}$ is a family of Hamiltonian structures such that $\omega_{0,1}=\omega$ and $\omega_{1,1}$ coincides with $\omega_2$ near $N$. Hence $\ker \omega_{1,1}|_N=\ker \omega_2|_N=e_*L$ which finishes the proof of the first statement taking $\hat \omega:=\omega_{1,1}$. \\

For part (2) we will need to show that there exists a two-form $\lambda$ with $[\lambda]=a$ for any $a\in H^2(M)$ such that $\lambda$ near some submanifold diffeomorphic to $N$ is exact. To prove the existence of such a two-form, consider a small disk $D^{2m+1}$ embedded in $M$. By Whitney's embedding theorem, there is some embedding of $N$ into $D^{2m+1}$. The relative cohomology exact sequence of $M$ with respect to $D^{2m+1}$ reads
\begin{center}
\begin{tikzcd}
H^1(D^{2m+1}) \arrow[rightarrow]{r}{} & H^2(M,D^{2m+1})\arrow[rightarrow]{r}{} & H^2(M)\arrow[rightarrow]{r}{} &  H^2(D^{2m+1})
\end{tikzcd}
\end{center}
We deduce that $H^2(M,D^{2m+1}) \cong H^2(M)$, and furthermore we know that 
$$H^2(M,D^{2m+1})\cong H_c^2(M\setminus D^{2m+1}).$$ 
Thus, for any class $a \in H^2(M)$ there is a two-form $\lambda$ with compact support in $M\setminus D^{2m+1}$ such that $[\lambda]=a$. Such a two-form trivially extends to $M$ and satisfies $\lambda=0$ near $N$. Arguing as in case (1), there is a maximal rank two-form $\tilde \omega$ which is exact near $N$ and whose kernel satisfies $\ker \tilde \omega=e_*L$ (where $e\colon N \rightarrow M$ is the embedding contained on a disk of $M$). Since we found a representative of $a$ which is exact near $N$, Theorem \ref{thm:McD} applies relative to a neighborhood of $N$, proving the existence of a Hamiltonian structure $\hat \omega$ such that $[\hat \omega]=a$ and $\ker \hat \omega$ coincides with $L$ along $e(N)$. 
\end{proof}

\begin{Remark}\label{rem:isotrop}
Observe that by construction that $e^*\hat \omega=0$ and $\tilde e^*\omega_a=0$.
\end{Remark}

From the previous theorem, we can realize arbitrary dynamics on a regular level set of a Hamiltonian vector field by a standard symplectization procedure. Let $X$ be a smooth non-vanishing flow on $N$. There is some embedding 
$$e\colon N\rightarrow S^{2m+1}, $$
with $2m+1 \in \{3n,3n+1\}$ by Whitney's embedding theorem. By Proposition \ref{thm:HamS} there is  exact Hamiltonian structure $d\gamma$ on $S^{2m+1}$ such that $e_*\langle X \rangle=\ker d\gamma|_{e(N)}$. Let $\tilde X$ be a section of $\ker d\gamma$ and $\lambda$ be a one-form such that 
\begin{equation*}
\begin{cases}
\lambda(\tilde X)=1\\
\lambda\wedge (d\gamma)^m\neq 0.
\end{cases}
\end{equation*}
Consider the two-form 
$$ \Omega= \lambda \wedge dt + d\gamma + td\lambda  $$
defined on $W=S^{2m+1}\times (-\varepsilon,\varepsilon)$, where $t$ denotes a coordinate in the second factor. It is an exact two form, and it is non-degenerate for $\varepsilon$ small enough. The Hamiltonian vector field of the function $f=t$ on the regular level set $\{t=0\}$ coincides with $\tilde X$, since:
\begin{align*}
\iota_{\tilde X} \Omega|_{t=0}&= -\iota_{\tilde X}\lambda dt + \iota_{\tilde X}d\gamma \\
					&=-dt.
\end{align*}
This gives another way of embedding arbitrary dynamics on regular energy level sets of Hamiltonian vector fields. As we did in Section \ref{ss:HamEx}, it would be possible to symplectically embed $W$ into some standard symplectic space of higher dimension and find a Hamiltonian vector field tangent to an arbitrary line field along a diffeomorphic copy of $N$.

\section{Hamiltonian embeddings and perturbations of the unit sphere}\label{s:Hamemb}

In the previous section, we motivated the existence of arbitrary dynamics on invariant submanifolds of Hamiltonian vector fields along a regular energy level set. In this section, we will study those embedded submanifolds in general and analyze their flexibility.

\subsection{Hamiltonian and generalized isosymplectic embeddings}

Let $N$ be an embedded submanifold on a symplectic manifold $(M,\omega)$. Given a line field $L$ in $N$, we are interested in understanding if there is some Hamiltonian function on $M$ with a regular energy level set containing $N$, and such that the Hamiltonian vector field is parallel to $L$ along $N$. This might be possible only after a small perturbation of the given embedding, as customary in $h$-principles for embeddings in symplectic geometry.

\begin{defi}\label{def:hamemb}
Let $L$ be a line field on a manifold $N$. An embedding $e\colon N\longrightarrow (M,\omega)$ into a symplectic manifold is called a Hamiltonian embedding of $(N,L)$ into $(M,\omega)$ if there is a Hamiltonian function $H$ and a component of a regular energy level set $W \supset N$ such that $\langle X_H \rangle|_{e(N)}=e_*L$.
\end{defi}

When we have  Hamiltonian embedding, any vector field spanning $L$ can be realized as the restriction of a Hamiltonian vector field along a regular energy level set.

\begin{lemma}\label{lem:sect}
Let $e\colon N\rightarrow (M,\omega)$ be a Hamiltonian embedding. Then for every vector field $X$ spanning $L$, there is a Hamiltonian function $\tilde H$ such that $X_{\tilde H}|_{e(n)}=e_*X$.
\end{lemma}
\begin{proof}
Assume that the level set is given by $H=0$. Fix a section $X$ of $L$. By hypothesis $X_H|_N=f.e_*X$ for some positive function $f \in C^\infty(N)$. Consider the function 
$$\tilde H=g H,$$
with $g$ any positive function such that $g=\frac{1}{f}$ near $N \subset \{H=0\}$. The Hamiltonian vector field defined by $\tilde H$ satisfies
\begin{align*}
\iota_{X_{\tilde H}} \omega_{std}=-d\tilde H= -Hdg-gdH.
\end{align*}
At $\{H=0\}$ we deduce that $\iota_{X_{\tilde H}} \omega_{std}|_{H=0}=-gdH$. On the other hand $\iota_{gX_{H}}\omega_{std}= g \iota_{X_H}\omega_{std}=-gdH$. Hence along $N$ we have $\iota_{gX_{H}} \omega_{std}|_N=-d\tilde H|_N$ which proves that $X_{\tilde H}|_N=g.X_H=e_*X$.
\end{proof}

To study the flexibility of Hamiltonian embeddings, we will need to fix an additional piece of information that is easily stated in terms of generalized isosymplectic embeddings. Generalized isosymplectic embeddings are defined by analogy to generalized isosymplectic immersions \cite[Section 16.5]{hprinc}. 

\begin{defi}
Let $\tilde \omega$ be a (not necessarily maximally non-degenerate) closed two-form on a closed manifold $N$. An embedding $e\colon N\longrightarrow (M,\omega)$ is called a generalized isosymplectic embedding of $(N,\tilde \omega)$ into $(M,\omega)$ if $e^*\omega=\tilde \omega$. 
\end{defi}
A formal solution to the generalized isosymplectic embedding differential relation allows for a tangential rotation.
\begin{defi}
Let $\tilde \omega$ be a (not necessarily maximally non-degenerate) closed two-form on a closed manifold $N$. An embedding $e\colon N\longrightarrow (M,\omega)$ is a formal generalized isosymplectic embedding if there is a family of monomorphisms $F_t\colon TN \rightarrow TM$ covering $e$ such that $F_0=de$ and $\omega\circ F_1=\tilde \omega$ and $e^*[\omega]=[\tilde \omega]$ holds for the cohomology classes.
\end{defi}
By abuse of notation, we will sometimes denote $\omega \circ F_1$ as $F_1^*\omega$. We define iso-Hamiltonian embeddings by analogy to iso-Reeb embeddings \cite{CMPP1}.
\begin{defi}\label{def:isoham}
Let $L$ be a line field on a manifold $N$ and a (not necessarily maximally non-degenerate) closed two-form $\tilde \omega$ such that $L \subseteq \ker \tilde \omega$. A generalized isosymplectic embedding of $(N, \tilde \omega)$ into $(M,\omega)$ is an iso-Hamiltonian embedding of $(N, L, \tilde \omega)$ if:
\begin{itemize}
\item $e(N)\subset W$ where $i\colon W\longrightarrow M$ is an embedded hypersurface,
\item $\ker i^*\omega|_{e(N)}=e_*L$.
\end{itemize}
\end{defi}
The formal counterpart requires only a formal isosymplectic embedding.
\begin{defi}
A formal iso-Hamiltonian embedding $e\colon (N,L,\tilde \omega)\rightarrow M$ is a formal isosymplectic embedding $(e,F_t)$ such that:
\begin{itemize}
\item $e(N)\subset W$ where $i\colon W\rightarrow M$ is a closed embedded hypersurface,
\item $\ker i^*\omega|_{e(N)}=F_1(L)|_W$.
\end{itemize}
\end{defi}
 As we did throughout this work, if $N$ is open we require that the embedding of $N$ is onto a proper compact submanifold of $W$. Finally, small isosymplectic and small iso-Hamiltonian embeddings will constitute the flexible subclass.
\begin{defi}\label{def:formalsmall}
A formal isosymplectic embedding is called \emph{small} if there is a conformal symplectic splitting $TM|_{e(N)}=V_1\oplus V_2$ such that $F_1(TN) \subsetneq V_1$. A formal iso-Hamiltonian embedding is defined analogously with the additional condition that $V_1$ is transverse to $W$.
\end{defi}
An analogous definition can be given for genuine small isosymplectic or iso-Hamiltonian embeddings. The simple observation here is that iso-Hamiltonian embeddings are, in particular, Hamiltonian embeddings. Their formulation in terms of generalized isosymplectic embeddings makes them suitable for an $h$-principle formulation in the next subsection.

\begin{lemma}\label{lem:isoHam}
For a given line field $L$ in $N$, if there is some two form $\tilde \omega$ such that $e\colon N\rightarrow (M,\omega)$ is an iso-Hamiltonian embedding of $(N,L,\tilde \omega)$, then $e$ is a Hamiltonian embedding of $(N,L)$.
\end{lemma}
\begin{proof}
By definition we have $i^*\omega|_{e(N)}=e_*L$. Let $H$ be any Hamiltonian function having $W$ as a regular level set. The Hamiltonian vector field $X_H$ of $H$ is parallel to $\ker i^*\omega$ along $W$, which implies that $X_H|_{e(N)}=e_*L$ and hence $e$ is a Hamiltonian embedding.
\end{proof}

\subsection{Flexibility of small iso-Hamiltonian embeddings}

We are now ready to show that small iso-Hamiltonian embeddings satisfy a full $h$-principle. We point out that the same proof applies to small generalized isosymplectic embeddings and hence they satisfy a full $h$-principle as well.

\begin{theorem}\label{thm:isoHam}
Small iso-Hamiltonian embeddings satisfy a full $h$-principle. It is $C^0$-dense and the hypersurface of the obtained genuine iso-Hamiltonian structure is a $C^0$-perturbation (compactly supported near the formal embedding of $N$) of the hypersurface $W$ given with the formal embedding.
\end{theorem}

\begin{proof}
We will first show that small iso-Hamiltonian embeddings satisfy the hypotheses of Theorem \ref{thm:hprincPw}, and defer to the appendix the fact that one can choose the hypersurface $C^0$-close to $W$. Let $e\colon (N,L,\tilde \omega) \longrightarrow (M,\omega)$ be a small formal iso-Hamiltonian embedding. That is, there is a hypersurface $i\colon W\longrightarrow M$ such that $W\supset e(N)$, and a family of monomorphisms $F_t\colon TN \longrightarrow TM$ such that $\ker i^*\omega|_{e(N)}=F_1(L)$, $\tilde \omega= \omega \circ F_1$ and $e^*[\omega]=[\tilde \omega]$. Furthermore, there is a symplectic splitting $TM|_M=V_1\oplus V_2$ such that $F_1(TN) \subsetneq V_1$, and we let $2k$ be the rank of $V_2$.

Extend $F_t$ to a family of isomorphisms
$$ G_t\colon TM|_{N} \rightarrow TM|_{N}, $$
which induces a family of symplectic bundle structures $(TM|_N, \omega_t= (G_t^{-1})^*\omega)$. We can assume that $\omega_t=\hat \omega_t \oplus \omega'$ for a fixed $\omega'$ defined on $V_2$, and the induced symplectic splitting is $TM|_N=V_1\oplus V_2$. 

\textbf{Step 1: An open submanifold of positive codimension and a symplectic form on it.}
Fixing a metric and using the exponential map along the fibers of $V_1$, we construct an open submanifold $E$ of codimension $2k$ such that $N\subset E$. The form $\hat \omega_1$ endows $TE$ with a symplectic bundle structure. By taking the exponential map for small enough times, we can safely assume that $E$ is transverse to $W$ so that $W'=W\cap E$ is a codimension one submanifold of $E$ satisfying $N\subset W'$. Denote by $j_1\colon W'\rightarrow E$ and $j_2\colon W'\rightarrow M$ the trivial inclusions of $W'$ into $E$ and $M$ respectively. The fact that $TE\oplus V_2$ is a symplectic orthogonal splitting implies that
$$\ker {j_1}^*\hat \omega_1|_N=e_*L,$$
and that
$${j_2}^*\omega_1|_{TN}={j_1}^*\hat \omega_1|_{TN}=\tilde \omega.$$
 Arguing exactly as in the proof of Proposition \ref{thm:HamS} using Lemma \ref{lem:Hamgerm}, we can assume up to a homotopy of non-degenerate two-forms that the restriction of $\hat \omega_1$ is a closed two-form in $W'$ and still satisfies $\ker {j_1}^*\hat \omega_1=e_*L$ and $\hat \omega_1|_{TN}=\tilde \omega$. Using \cite[Lemma 16.5.2]{hprinc} in the normal bundle of $W'$, we construct a germ of symplectic structure near $W'$ inside $E$ that coincides with $\hat \omega_1$ in $W$. Concretely, we construct a homotopy of non-degenerate two-forms $\hat \omega_t$ with $t\in[1,2]$ such that $\hat \omega_2$ is symplectic in $E$ (up to shrinking $E$ a bit) and satisfies $\hat \omega_2|_N= \hat \omega_1|_N $ and hence $\ker {j_1}^*\hat \omega_2|_N=e_*L$ and $\hat \omega_2|_{TN}=\tilde \omega$. In particular, in the symplectic manifold $(E,\hat \omega_2)$, the embedding $e$ restricted in the target to $E$ is a genuine iso-Hamiltonian embedding of $(N,L,\tilde \omega)$.\\
 
\textbf{Step 2: A formal isosymplectic embedding.}  We claim that the symplectic manifold $(E, \hat \omega_2)$ is formally isosymplectically embedded in $(M,\omega)$. Indeed, the family $\hat \omega_t$ defined for $t\in [0,2]$ in $TE$ defines by Lemma \ref{lem:iso} a family of isomorphisms
$$ R_t'\colon TE \longrightarrow TE $$
such that $\hat \omega_t=\hat \omega_0 \circ R_t'$. It extends to a family of monomorphisms
$$ R_t \colon TE \longrightarrow TM $$
such that $\hat \omega_t= \omega_0 \circ R_t$, where $\omega_0=\hat \omega_0 \oplus \omega'=\omega$, the ambient symplectic structure.  We have thus proved that small iso-Hamiltonian embeddings admit an isosymplectic realization (of codimension $\rank V_2=2k\geq 2$) as in Definition \ref{def:isosympext}, and one easily checks that the construction is parametric and relative to the domain in the sense that $R_t$ is constant if the iso-Hamiltonian embedding was already genuine near a closed set of $N$.\\

\textbf{Step 3: A hypersurface $C^0$-close to $W$.} To conclude by applying Theorem \ref{thm:hprincPw}, it remains to check that the relation behaves well under inclusions. This is easily checked if we only require the existence of a germ of a hypersurface in Definition \ref{def:isoham}, by using Weinstein's tubular neighborhood theorem for symplectic submanifolds to extend the hypersurface $W'$ making an embedding iso-Hamiltonian in the symplectic submanifold $E$ to a suitable germ of a hypersurface in $M$. 
 To find a globally defined hypersurface making the final embedding a genuine iso-Hamiltonian one, an option is to carefully follow the argument in the proof of Theorem \ref{thm:hprincPw} to see how to construct it out of the original hypersurface $W$ given by the formal iso-Hamiltonian embedding structure. The advantage of this approach is that it allows showing that the final hypersurface can be chosen to be $C^0$-close to $W$. Since this is simply a technical lemma that requires an inspection of the proof of Theorem \ref{thm:hprincPw}, we defer it to the Appendix, see Lemma \ref{lem:appHyper}.
\end{proof}

\subsection{Dynamics on $C^0$-perturbations of the unit sphere}

In this last section, we apply the flexibility of Hamiltonian embeddings to prove Theorem \ref{thm:mainuniv}: perturbations of the unit sphere contain arbitrary dynamics of high enough codimension. In the proof, we use the preliminary result Proposition \ref{thm:HamS} to show that a given embedding is a formal small iso-Hamiltonian embedding in this codimension, and conclude applying the $h$-principle.

\begin{proof}[Proof of Theorem \ref{thm:mainuniv}]
Denote by $L$ the line field over $N$ generated by $X$. Let $2k-1$ be the smallest odd integer in $\{3\dim N,3\dim N+1\}$. Denote by $S^{2k-1}$ the unit sphere in $(M=\R^{2k},\omega_{std})$. The standard symplectic structure induces on $S^{2k-1}$ the Hamiltonian structure $d\alpha$, where $\alpha$ is the stardard contact form on $S^{2k-1}$. Let $e\colon N \rightarrow S^{2k-1}$ be any smooth embedding (at least one exists by Whitney's embedding theorem). Applying Proposition \ref{thm:HamS} there is a homotopy of Hamiltonian structures $\tilde \omega_t$ such that $\tilde \omega_0=d\alpha$ and $\tilde \omega_1$ satisfies $e_*L=\ker \tilde \omega_1|_{e(N)}$. This homotopy induces a homotopy of symplectic hyperplane bundles $(\xi_t,\hat \omega_t)$ over $S^{2k-1}$, where $(\xi_0,\omega_0)=(\xi_{std},d\alpha|_{\xi_{std}})$ (the standard contact structure on $S^{2k-1}$ seen as a codimension one symplectic bundle). It is induced by a family of isomorphisms $\phi_t\colon \xi_0\rightarrow \xi_t$ satisfying $\hat\omega_t= (\phi_t^{-1})^*\ \omega_0$. \\

Extend it to a family of rank $2k$ symplectic bundles as follows. The standard symplectic form $\omega_{std}$ induces a symplectic bundle structure on $TM|_{S^{2k-1}}$ which decomposes as $(V_0\oplus \xi_0, v_0 \oplus d\alpha|_{\xi_0})$. The rank $2$ subbundle $V_0$ is the symplectic orthogonal complement of $\xi_0$ with respect to $\omega_{std}$. By construction, the subbundle $V_0$ decomposes as $\ker d\alpha \oplus Y$ where $Y$ is the radial line bundle defined near the sphere and transverse to it. We can extend the isomorphisms $\phi_t$ to a family of isomorphisms
$$ G_t\colon TM|_{S^{2k-1}} \rightarrow TM|_{S^{2k-1}}, $$
such that $G_t(\xi_0)=\xi_t$, $G_t(\ker d\alpha)=\ker \tilde \omega_t$ and $G_t(Y)=Y$. Define $V_t= \ker \tilde \omega_t \oplus Y$, then 
$$(TM|_{S^{2k-1}},\omega_t)=\big(V_t \oplus \xi_t, v_t \oplus \hat \omega_t \big),$$ 
is a family of symplectic bundle structures on $TM|_{S^{2k-1}}$, where $v_t=(G_t^{-1})^*v_0$. The path of non-degenerate two-forms $\omega_t$ with $t\in [0,1]$ satisfies $\omega_0=\omega_{std}$ and $\omega_1$ restricted to $S^{2k-1}$ is $\hat \omega_1$.

Identify a neighborhood $U$ of $N$ with a neighborhood of the zero section of the normal bundle $E$ of $N$ inside $\mathbb{R}^{2k}$. It is endowed with the non-degenerate two-form $\omega_1$. Denote by $\tilde e$ the embedding
$$ \tilde e\colon N\rightarrow \mathbb{R}^{2k}, $$
obtained by composing the embedding $e$ with the trivial inclusion $i\colon S^{2m-1}\rightarrow \mathbb{R}^{2m}$ of $S^{2m-1}$ into $\mathbb{R}^{2k}$. By Remark \ref{rem:isotrop}, it satisfies
$$\tilde e^*\omega_1=e^*i^*\omega_1=e^*\tilde \omega_1=0.$$
 The embedding $e$ is then a formal iso-Hamiltonian embedding of $(N,L,0)$ into $\R^{2k}$, where the hypersurface $W$ is the unit sphere $S^{2m-1}$. To make it a small formal iso-Hamiltonian embedding, we just consider the trivial codimension two isosymplectic embedding
$$j\colon (\mathbb{R}^{2k},\omega_{std}) \hookrightarrow (\mathbb{R}^{2k} \times \mathbb{R}^2, \omega_{std}\oplus dx\wedge dy),$$
where $x,y$ are coordinates in the second factor of $\mathbb{R}^{2k} \times \mathbb{R}^2$.
Then $j\circ e$ is a formal iso-Hamiltonian embedding, that is small when taking as hypersurface the unit sphere in $\mathbb{R}^{2k+2}$. Applying Theorem \ref{thm:isoHam}, it is homotopic to a $C^0$-close genuine iso-Hamiltonian embedding. The hypersurface $W$ is given by a $C^0$-perturbation of the unit sphere. We conclude using Lemmas \ref{lem:isoHam} and \ref{lem:sect}.
\end{proof}

\begin{Remark}\label{rm:generalHam}
As in \cite[Theorem 1.8]{CMPP1}, the previous theorem admits a completely general statement. Using obstruction theory, it can be shown that any smooth embedding $e\colon N\rightarrow (M,\omega)$ contained in a hypersurface of a symplectic manifold $(M,\omega)$ is a small formal iso-Hamiltonian embedding $e_X\colon (N,X,0)\rightarrow (M,\omega)$, as long as $\dim M\geq 3\dim N+3$. Then by Theorem \ref{thm:isoHam}, the embedding $e_X$ is isotopic to a $C^0$-close genuine (small) iso-Hamiltonian embedding.
\end{Remark}

\appendix

\section{Stability lemma and isotropic subbundles}

In this appendix, we recall a standard stability lemma for homotopies of vector bundles, and computations in algebraic topology done in \cite{CMPP1}. For the stability lemma, see for instance \cite[Section 5.2]{CMPP1}.
\begin{lemma}\label{lem:iso}
Let $V_{t}$ be a parametric family of real bundles over a fixed smooth manifold $M$ parametrized by $t \in [0,1]$. Then, there exists a family 
$$\phi_{t}\colon V_{0} \to V_{t}$$
of bundle isomorphisms. If the $V_t$ are equipped with a family of fiberwise symplectic structures $\omega_t$, then we can choose $\phi_t$ to be symplectic bundle isomorphisms.
\end{lemma}

The following lemma was proved in \cite{CMPP1}.

\begin{lemma}[Lemma 5.13 in \cite{CMPP1}] \label{lem:isotro}
Let $(\xi,\omega)$ be a symplectic bundle of rank $2m$ over $N$ and denote $\eta= TN \cap \xi$ which is of rank $n-1$. If $2m\geq 3n-1$ then $\eta$ is homotopic to an isotropic subbundle of $(\xi,\omega)$.
\end{lemma}
We will use this lemma to prove the following technical result, which is implicitly proved in \cite{CMPP1}. We state it properly and prove it for the sake of completeness.

\begin{prop}\label{prop:isotr}
Let $(\hat \xi, \hat \omega)$ be a symplectic hyperplane distribution on $M^{2n+1}$, $N$ a submanifold such that $\dim M\geq 3\dim N$ and $\eta$ an hyperplane distribution on $N$. Then $(\hat \xi, \hat \omega)$ is homotopic through symplectic hyperplane distributions to $(\xi,\omega)$ satisfying that $\xi|_N \cap TN=\eta$ is an isotropic subbundle of $\xi$.
\end{prop}

\begin{proof}
Otherwise stated, we consider all distributions and tangent bundles over $N$, i.e. $TM$ denotes $TM|_N$.
Let $Y$ be some line field such that $TM=\hat \xi \oplus Y$. We first prove that $Y$ is homotopic to $L$ (a line field transverse to $\eta$) along the embedded submanifold $e(N)$ (that we denote by $N$ by abuse of notation). 

We want to find a family of bundle isomorphisms
$$ F_t\colon TM \rightarrow TM $$
such that $e$ is the induced map on the base and $(F_1)_*Y=L$. To find such a family, we will show that there is a path between any two sections of an $S^{2m}$ bundle (the unit tangent bundle of $M$ with respect to any metric) over a base space of dimension $n$. A sufficient condition is that 
$$ \pi_k (S^{2m})=0 \text{ for } k=0,...,n. $$
This is true as long as $n<2m$ which is clearly satisfied, since by hypothesis $3n-1 \leq 2m$. Hence, there is a path of non-vanishing sections $X_t$ of $TM$ such that $X_0=Y$ and $X_1=L$. Applying Lemma \ref{lem:iso}, we find a family of isomorphisms that can be extended to $TM$:
$$F_t\colon TM \rightarrow TM $$
such that $F_1(Y)=X$. The distribution 
$$\xi'=F_1(\xi)$$
is a symplectic vector bundle $(\xi',\omega'=(F_1^{-1})^*\hat \omega)$ of corank $1$ in $TM$ satisfying that  $\xi' \pitchfork X$. A linear interpolation $\eta_t$ between $\eta_0=\eta$ and $\eta_1=\xi'\cap TN$ is a well defined homotopy of subbundles of rank $n-1$ along $N$ since both $\eta$ and $\xi'\cap TN$ are both transverse to $X$. Using an auxiliary metric on $\xi'$, we can extend the homotopy to a homotopy of rank $2m$ complex bundles
$$ \xi_t= \eta_t \oplus (\eta_0)^\perp,  $$
where $(\eta_0)^\perp$ is the orthogonal of $\eta_0$ in $\xi'$ with respect to the auxiliary metric. This is a well-defined homotopy because $\eta_t$ lies in $TN$ for all $t$, which ensures that $\xi_t$ has constant rank $2m$. By Lemma \ref{lem:iso} there is a family of isomorphisms $\varphi_t\colon \xi_0 \rightarrow \xi_t$ such that $(\xi_t, w_t=(\varphi_t^{-1})^*\hat \omega)$ is a homotopy of symplectic hyperplane distributions. The distribution $\xi_1$ satisfies $\xi_1\cap TN=\eta$.

We now apply Lemma \ref{lem:isotro}, and find a family of subbundles $\eta_t\subset \xi_1$ such that $\eta_0=\eta$ and $\eta_1$ is an isotropic subbundle of $(\xi_1,\omega_1)$. Arguing as before, we use Lemma \ref{lem:iso} and extend it to a family of isomorphisms 
$$ \tilde \varphi_t\colon \xi_1\rightarrow \xi_t, \text{ for } t\in [1.2] $$
such that $(\xi_1,(\tilde \varphi_2^{-1})^*\omega_1)$ satisfies that $\eta$ is an isotropic subbundle. We extend the family $(\xi_t,\omega_t)$ for values of the parameter $t$ in $[1,2]$ as 
$(\xi_t,\omega_t)=(\xi_1,(\tilde \varphi_{t}^{-1})^*\omega_1).$

To conclude, we use the homotopy extension property to extend all the family $(\xi_t,\omega_t)$ of symplectic distributions of codimension to one defined over all the ambient manifold $M$. It satisfies $(\xi_0,\omega_0)=(\hat \xi, \hat \omega)$ and $\xi_2|_N\cap TN=\eta$ is an isotropic subbundle of $(\xi_2, \omega_2)$.
\end{proof}

\section{Germs of Hamiltonian structures}

 In this appendix, we include for completeness a proof of a variation of \cite[Lemma 16.2.2]{hprinc} for Hamiltonian structures.

\begin{lemma}\label{lem:Hamgerm}
Let $\eta=\ker \alpha$ be a tangent hyperplane distribution on a manifold $N$ and $\pi\colon E\rightarrow N$ a vector bundle over $N$. Let $\hat \eta$ be the codimension one subbundle of $TE|_N$ which is the direct sum of $\eta$ with $E$. Suppose that there exists a not necessarily closed two-form $\hat \omega$ of maximal rank in $E$, non-degenerate in $\hat \eta$, and such that $\hat \omega|_{TN|_N}=\tilde \omega$, a closed two-form in $N$. Then there exists a closed form $\omega$  of maximal rank defined in a small neighborhood of $N$ in $E$, non-degenerate in $\hat \eta$ and such that $\omega|_{TE|_N}=\hat \omega|_{TE|_N}$.
\end{lemma}

\begin{proof}
Let $2m+1$ be the dimension of $E$, and $n$ the dimension of $N$. Fix a covering $U_j$ of $N$ trivializing $E$, in a way that $V_j\cong U_j \times \mathbb{R}^{2m-n}$ is a covering of $E$. Since $\hat \omega$ restricts as $\tilde \omega$ in $N$, we can write it as
$$ \hat \omega=\pi^*\tilde \omega + \beta, $$
where $\beta$ is some two-form such that $\beta|_{TN|_N}=0$. Let $t_1,...,t_k$ be coordinates in the fibers of $E$, where $k=2m+1-n$. Then $\beta$ writes as
$$\beta= \sum_{i=1}^k dt_i \wedge \beta_i + \beta' , $$
for some one-forms $\beta_1,...,\beta_k$, and a one-form $\beta'$ such that $\beta'|_{TE|_N}=0$.

Let $\chi_j$ be a partition of unity subordinated to $U_j$, and on each $U_j$ consider the two-form 
$$ \gamma_j= d( \chi_j \sum_{i=1}^k t_i \beta_i). $$
Observe that $\omega_j=\pi^*\tilde \omega+\gamma_j$ is closed, and by construction $\omega_j|_{TE|_{U_j}}=\hat \omega|_{TE|_{U_j}}$. In particular, it is necessarily of maximal rank near $U_j$ and non-degenerate in $\hat \eta|_{U_j}$. The two-form $\omega=\pi^*\tilde \omega+\sum_j \gamma_j$ is closed, satisfies $\omega|_{TE|_N}=\hat \omega|_{TE|_N}$ and thus in a neighborhood of $N$ it is of maximal rank and non-degenerate in $\hat \eta$.
\end{proof}

\section{$C^0$-close hypersurface for iso-Hamiltonian embeddings}

In this appendix, we follow the argument of Theorem \ref{thm:hprincPw} in the context of the proof of Theorem \ref{thm:isoHam}.

\begin{lemma}\label{lem:appHyper}
In Theorem \ref{thm:isoHam}, the formal iso-Hamiltonian embedding $e$ (with hypersurface $W$) is isotopic to a genuine iso-Hamiltonian embedding whose hypersurface is $C^0$-close to $W$.
\end{lemma} 

\begin{proof}
Following the notation of the proof of Theorem \ref{thm:isoHam}, let us show that the hypersurface making $e_1|_N$ a genuine iso-Hamiltonian embedding is a global closed hypersurface that can be taken $C^0$-close to the hypersurface $W$ given in the formal embedding $e$. 
 
The manifold $E$ is open, so we can apply Theorem \ref{thm:isosymp} and find $\tilde e_t$ such that $\tilde e_0$ is just the trivial embedding of $E$ and $\tilde e_1$ is isosymplectic and $C^0$-close to $\tilde e_0$ near $N$ (which is a core of $E$). So, up to taking a smaller neighborhood of $N$ inside $E$ we can assume that $\tilde e_t$ is $C^0$-small. We claim that $\tilde e_1:=e_1|_N$ is a genuine small iso-Hamiltonian embedding. For this, it only remains to construct a hypersurface $\tilde W$ such that the kernel of $\omega$ along $N\subset \tilde W$ coincides with $(\tilde e_1)_*(L)$. The embedded submanifold $e_1(N)$ lies in $e_1(W')$, a hypersurface in $e_1(E)$. Furthermore, close to $N$ the submanifold $e_1(W')$ is $C^0$-close to $W'$. In a symplectic neighborhood of the symplectic submanifold $e_1(E)$ we consider an extension of $W'$ to a hypersurface $\tilde W$ as follows. By the Weinstein tubular neighborhood theorem, a neighborhood $U$ of $e_1(E)$ is given by a neighborhood of the zero section of its symplectic normal bundle. Hence $U$ is endowed with a symplectic form that splits as
$$\pi^*\omega_2+ \Omega,$$
where $\pi$ denotes the projection $\pi\colon U\rightarrow E$, and $\Omega$ global closed two-form that is fiberwise symplectic in the disk fibers and satisfies $\Omega|_E=0$. Let $h$ be a function on $E$ such that $e_1(W')$ is given by a regular connected component of $\{h=0\}$. Let $\tilde W$ be the hypersurface given by $\pi^*h=0$. Looking at $e_1$ as an ambient isotopy, the submanifold $e_1(W')$ lies in the hypersurface $e_1(W)$. In particular, we can slightly deform $e_1(W)$ by an isotopy compactly supported near $N$ so that it coincides with $\pi^*h=0$ near $e_1(N) \subset e_1(W')$. Summarizing, we constructed a closed hypersurface $i\colon \tilde W \longrightarrow M$ satisfying $\ker i^*\omega|_{e_1(N)}={\tilde {e_1}}_*L$ that is $C^0$-close to $e_1(W)$. The hypersurface $\tilde W$ is just a $C^0$-small perturbation (with support near $N$) of the hypersurface $W$ given in the formal iso-Hamiltonian embedding. 
\end{proof}

\end{document}